\documentclass{article}
\usepackage{graphicx} 

\title{\textbf{Deformations of Non-Kähler Hyperbolicity Notions and Modifications of Degenerate Balanced Manifolds}}
\author{ABDELOUAHAB KHELIFATI}
\date{}
\usepackage{graphicx} 
\usepackage{booktabs}
\usepackage[english]{babel}
\usepackage{amsmath}
\usepackage{amssymb}
\usepackage{amsfonts}
\usepackage{amsthm}
\usepackage{mathrsfs}
\usepackage{enumitem}
\usepackage{extpfeil}
\usepackage{extarrows}
\usepackage{musicography}
\usepackage[utf8]{inputenc}
\usepackage[T1]{fontenc}
\usepackage[parfill]{parskip}
\usepackage{hyperref}
\hypersetup{colorlinks=true, linkcolor=blue, citecolor=red, urlcolor=blue}
\usepackage{lipsum}
\usepackage{mathtools}
\usepackage{tikz-cd}
\usepackage[all]{xy}

\usepackage{makeidx}

\renewcommand{\thesubsection}{\arabic{section}.\arabic{subsection}}
\usepackage{fullpage}
\usepackage{amssymb, amsmath, amsthm, amscd}
\usepackage{titlesec}
\titleformat{\subsection}
  {\normalfont\Large\bfseries}
  {\thesubsection}{1em}{}
\numberwithin{equation}{section}

\theoremstyle{plain}
\newtheorem{theorem}{Theorem}[section]
\newtheorem{dfn}[theorem]{Definition}
\newtheorem{thm}[theorem]{Theorem}
\newtheorem*{thmA}{Theorem A}
\newtheorem*{thmB}{Theorem B}
\newtheorem*{thmC}{Theorem C}
\newtheorem*{thmD}{Theorem D}
\newtheorem*{thmE}{Theorem E}
\newtheorem{prop}[theorem]{Proposition}
\newtheorem{lem}[theorem]{Lemma}
\newtheorem{cor}[theorem]{Corollary}
\newtheorem{conj}[theorem]{Conjecture}
\theoremstyle{definition}

\newtheorem{rmk}[theorem]{Remark}
\newtheorem{obs}[theorem]{Observation}
\newtheorem{ex}[theorem]{Examples}
\newtheorem{pbm}{Problem}
\newcommand{\B}{\mathbb{B}}
\newcommand{\C}{\mathbb{C}}

\newcommand{\G}{\mathbb{G}}
\newcommand{\CP}{\mathbb{P}}
\newcommand{\R}{\mathbb{R}}
\newcommand{\Sp}{\mathbb{S}}

\begin{document}

\maketitle

\vspace{5ex}

\large{\textbf{Abstract.}} We study deformation properties of balanced hyperbolicity, with a particular emphasis on degenerate balanced manifolds and their behavior under smooth modifications. 

\hspace{1ex}From a different perspective, we introduce two new notions of hyperbolicity for compact complex non-Kähler manifolds $X$ of complex dimension $\dim_{\mathbb{C}}X=n$, in general degree $2p$ with $1 \leq p \leq n-1$. These notions are motivated by the work of D.~Popovici and H.~Kasuya on partial hyperbolicity in arbitrary degree and by the work of F.~Haggui and S.~Marouani on $p$-Kähler hyperbolicity. The first notion, called \emph{$p$-SKT hyperbolicity}, extends SKT hyperbolicity and Gauduchon hyperbolicity to degree $2p$. Similarly, the second notion, called \emph{$p$-HS hyperbolicity}, generalizes the notion of strongly Gauduchon hyperbolicity introduced by Y.~Ma.

\hspace{1ex}We then analyze the relationships between these analytic notions and geometric notions of hyperbolicity, namely Brody/Kobayashi hyperbolicity and \emph{$p$-cyclic hyperbolicity} in degree $2p$ for $2 \leq p \leq n-1$. In addition, we study the behavior of $p$-HS hyperbolicity and $p$-Kähler hyperbolicity under holomorphic deformations, establishing openness results for these properties.

\section{Introduction}
\hspace{1ex}In one of his seminal papers \cite{gromov1991kahler}, M. Gromov introduced a notion of hyperbolicity for compact Kähler manifolds, called \emph{Kähler hyperbolicity}, defined in terms of a boundedness property relative to Kähler metrics. More precisely, let $(X,\omega)$ be a compact Hermitian manifold with $\dim_{\mathbb{C}}X=n$. A $k$-form $\alpha$ on $X$ is said to be $\widetilde{d}$(\emph{bounded}) if its pullback $\widetilde{\alpha}:=\pi^*\alpha$ to the universal cover $\pi: \widetilde{X} \longrightarrow X$ satisfies $\widetilde{\alpha}=d\beta$ for some $\widetilde{\omega}$-bounded $(k-1)$-form $\beta$ on $\widetilde{X}$. This notion clearly does not depend on the choice of the Hermitian metric on $X$, since all Hermtitian metrics on a compact complex manifold are comparable. The manifold $X$ is called \emph{Kähler hyperbolic} if it admits a Kähler metric $\omega$ that is $\widetilde{d}$(bounded).

\hspace{1ex}Over the last few years, several authors have proposed extensions of this hyperbolicity notion beyond the Kähler setting in several ways. First, D. Popovici and S. Marouani introduced the notion of \emph{balanced hyperbolicity} in \cite{marouani2023balanced} as a generalization of Gromov's Kähler hyperbolicity: a compact complex manifold $X$ is said to be \emph{balanced hyperbolic} if it admits a $\widetilde{d}$(bounded) balanced metric $\omega$. This class includes, for example, Kähler hyperbolic manifolds and another special class called \emph{degenerate balanced} manifolds). Further generalizations in the same spirit include \emph{SKT hyperbolicity} and \emph{Gauduchon hyperbolicity}, defined by S. Marouani in \cite{marouani2023skt}, as well as \emph{sG hyperbolicity}, introduced by Y. Ma in \cite{ma2024strongly}.

\hspace{1ex}The main goal of this paper is to study the deformation stability of such non-Kähler hyperbolicity notions, and to introduce new hyperbolicity concepts that fit naturally into this framework. But first, we obtain the following result concerning the behavior of degenerate balanced manifolds under modifications:
\begin{thmA}
    Let $f:\widehat{X}\longrightarrow X$ be a smooth \textbf{surjective} modification of compact complex manifolds, with $\dim_\C X=n\geq 3$. If $\widehat{X}$ is degenerate balanced, then $X$ is degenerate balanced.
\end{thmA}
\hspace{1ex}We then investigate the openness under holomorphic deformations of balanced hyperbolicity and of the degenerate balanced class of manifolds:
\begin{thmB}
    Let $\{X_t:t\in\B_\varepsilon(0)\}$ be an analytic family of compact complex $n$-dimensional manifolds such that $X_0$ is balanced. Assume that the fibers $X_t$ satisfy the weak $(n-1,n)$-$\partial\bar\partial$-property for $t\neq 0$.
\begin{enumerate}
    \item If $X_0$ is balanced hyperbolic, the fibers $X_t$ carry balanced metrics $\omega_t$ that are Gauduchon hyperbolic for small $t$.
    \item If moreover, $X_0$ is degenerate balanced and the fibers $X_t$ satisy the mild $(n,n-2)$-$\partial\bar\partial$-property for $t\neq 0$, the fibers $X_t$ are degenerate balanced for small $t$.
\end{enumerate}
\end{thmB}
\hspace{1ex} When the family $\{X_t:t\in\B_\varepsilon(0)\}$ satisfies the $\partial\bar\partial$-property, we get the following stronger statement, which is new in the context of degenerate balanced manifolds:
\begin{thmC} Let $\{X_t:t\in\B_\varepsilon(0)\}$ be an analytic family of compact $\partial\bar\partial$-manifolds of complex dimension $n$ such that $X_0$ is balanced.
\begin{enumerate}
    \item Assume $X_0$ is balanced hyperbolic, and let $\omega$ be a balanced hyperbolic metric on $X_0$. Then $\omega$ extends to a $\mathscr{C}^\infty$-family of balanced metrics $\{\omega_t\}_{t\sim 0}$ that are Gauduchon hyperbolic on the respective fibers $X_t$ for small $t$.
    \item If moreover, $\omega$ is degenerate balanced, the $\mathscr{C}^\infty$-family of balanced metrics $\{\omega\}_{t\sim0}$ on the fibers $\{X_t\}_{t\sim 0}$ are degenerate balanced with a $\mathscr{C}^\infty$-family of $d$-potentials $\{w_t\}_{t\sim 0}$ for small $t$, i.e. $\omega=dw_0$ and $\omega_t=dw_t$, $\forall t\sim 0$.
\end{enumerate}
\end{thmC}
\hspace{1ex}Other notions of hyperbolicity were introduced by D.~Popovici and H.~Kasuya in \cite{kasuya2025partially}, where they initiated the concept of \emph{partial hyperbolicity}, which constitutes the first hyperbolicity notion formulated in intermediate degree $2p$ with $1 \leq p \leq n-1$. In a related direction, a natural way to generalize the previously mentioned \emph{metric} notions of hyperbolicity is to impose analogous conditions in bidegree ($p,p$). This approach was developed by S.~Marouani and F.~Haggui in \cite{haggui2023compact}, where they introduced the notion of \emph{$p$-K\"ahler hyperbolicity} as a degree $p$ generalization of K\"ahler hyperbolicity.

\hspace{1ex}In this paper, we introduce the notion of \textbf{Hermitian-symplectic} (or \textbf{HS}) \textbf{hyperbolicity}, in which the Kähler condition is replaced by the Hermitian-symplectic condition. This notion is motivated in part by a question raised by J. Streets and G. Tian on the existence of non-Kähler Hermitian–symplectic manifolds (\cite{streets2010parabolic}, Question 1.7). We investigate the similarities and differences between hyperbolicity properties in the Hermitian–symplectic and Kähler settings. We also define degree $2p$ analogues of this notion and of SKT hyperbolicity, namely \textbf{\emph{p}-SKT hyperbolicity} and \textbf{\emph{p}-HS hyperbolicity}, and we prove the following results:
\begin{thmD}
Let $X$ be a compact complex manifold of complex dimension $\dim_\C X=n\geq 2$, and let $1\leq p\leq n-1$. Then:
\begin{enumerate}
    \item If $X$ is HS hyperbolic, then $X$ is Kobayashi hyperbolic.
    \item If $X$ is $p$-HS hyperbolic or $p$-SKT hyperbolic, then $X$ is $p$-cyclically hyperbolic $($or $p$-hyperbolic in the sense of Definition 2.5, \cite{kasuya2025partially}$)$.
\end{enumerate}
\end{thmD}
We also establish the following deformation results for $p$-HS hyperbolic manifolds and $p$-Kähler hyperbolic manifolds:
\begin{thmE}
Let $\{X_t\hspace{1pt}:\hspace{1pt}t\in\B_\varepsilon(0)\}$ be an analytic family of compact complex $n$-dimensional manifolds, where $\B_\varepsilon(0)\subset\C^N$ is the open ball of radius $\varepsilon>0$ centered at the origin, and fix $1\leq p\leq n-1$. Then:
\begin{enumerate}
    \item If $X_0$ is $p$-HS hyperbolic, then the fibers $X_t$ are $p$-HS hyperbolic for all $t$ sufficiently close to $0$.
    \item If $X_0$ is a $p$-Kähler hyperbolic $\partial\bar\partial$-manifold, then for sufficiently small $t$, the fibers $X_t$ are $p$-SKT hyperbolic.
\end{enumerate}
\end{thmE}

\section{Balanced hyperbolic manifolds}
Let $X$ be a compact connected complex manifold of complex dimension $\dim_\C X=n\geq 2$, and $\pi:\widetilde{X}\longrightarrow X$ its universal cover. Recall that a Hermitian metric $\omega$ on $X$ is \textbf{semi-Kähler} (or \textbf{balanced}, or \textbf{co-Kähler}) if $d\omega^{n-1}=0$ (\cite{gauduchon1977fibres},\cite{michelsohn1982existence}). 
\begin{dfn}[\textbf{Balanced hyperbolicity}, \cite{marouani2023balanced}] We say that a compact complex manifold $X$ is \textbf{balanced hyperbolic} if it admits a balanced metric $\omega$ such that $\omega^{n-1}$ is $\widetilde{d}($bounded$)$, i.e. there exists a $\widetilde{\omega}$-bounded $(2n-3)$-form $\alpha$ on $\widetilde{X}$ such that $\widetilde{\omega}^{n-1}:=\pi^*\omega^{n-1}=d\alpha$.
\end{dfn}
There are mainly four known classes of Balanced hyperbolic manifolds, namely:
\begin{enumerate}
    \item Kähler hyperbolic manifolds (thanks to Lemma 2.2 in \cite{marouani2023balanced}).
    \item \textbf{Degenerate balanced} manifolds, i.e. $n$-dimensional compact complex manifolds that carry a Hermitian metric $\omega$ such that $\omega^{n-1}$ is $d$-exact (see \cite{popovici2015aeppli}, Proposition 5.4). Obviously, this class of manifolds does not contain any Kähler ones (if a compact complex manifold $X$ admits a Kähler metric $\omega$, then $\{\omega^{n-1}\}\neq 0\in H^{2n-2}_{DR}(X,\R)$). Examples of such manifolds appear starting from complex dimension 3. Two classes of examples of degenerate balanced manifolds are known:
    \begin{enumerate}
        \item The connected sums $X=\#_k(\Sp^3\times\Sp^3)$ of $k\geq 2$ copies of $\Sp^3\times\Sp^3$ endowed with the Friedman-Lu-Tian complex structure constructed via \textit{conifold transitions} (see \cite{friedman1991threefolds}, \cite{lu1996complex} and \cite{fu2012balanced}).
        \item The \textbf{Yachou manifolds} $X=G/\Gamma$, where $G$ is a \textit{semi-simple} complex Lie group and $\Gamma\subset G$ a lattice (\cite{yachou1998varietes}).
    \end{enumerate}
    \item Products of (1) and (2) (thanks to Proposition 2.17, \cite{marouani2023balanced}), which need not be Kähler nor degenerate balanced.
    
    \item Via smooth modifications of the previous examples. More precisely: Take a smooth modification $f:\widehat{X}\longrightarrow X$ of compact complex manifolds,  with $\dim_\C X=n\geq 3$. Assume that either $\pi_i(\widehat{X})=0$ for $2\leq i\leq 2n-3$, or $f$ is a \textit{contraction to points}, i.e. $f$ restricts to a biholomorphism outside of the preimage of a countable set of points in $X$. Then the following holds:
    \begin{center}
        $\widehat{X}$ is balanced hyperbolic $\xLongrightarrow{\qquad}$ $X$ is balanced hyperbolic
    \end{center}
    This has been proved in \cite{fu2024birational} (see Theorem 1.3.(1) and Theorem 1.4). 
\end{enumerate}
\subsection{Modifications of Compact Degenerate Balanced Manifolds}
It is natural to ask whether the degenerate balanced property is preserved under smooth modifications. Toward answering this question, we first make the following observation:
\begin{obs} Let $X$ be a compact balanced hyperbolic manifold with $\dim_\C X=n$.
\begin{enumerate}
    \item If $X$ is degenerate balanced, then there are no complex hypersurfaces in $X$.
    \item If $X$ contains a complex hypersurface, then $\pi_1(X)$ is infinite
\end{enumerate} \label{97}
\end{obs}
\begin{proof}
\begin{enumerate}
    \item Assume $X$ is degenerate balanced and $\omega$ is a degenerate balanced metric on $X$, that is $\omega^{n-1}=d\alpha$ for some smooth ($2n-3$)-form $\alpha$ on $X$. Assume there exists a complex hypersurface $H\subset X$, then:
    \begin{equation*}
        \text{Vol}_\omega(H)=\dfrac{1}{(n-1)!}\int_H\omega^{n-1}=\dfrac{1}{(n-1)!}\int_H d\alpha=\dfrac{1}{(n-1)!}\int_{\partial H}\alpha=0 \quad(\text{since }\partial H=\emptyset)
    \end{equation*}
    This is a contradiction, hence there are no complex hypersurfaces in $X$.
    \item Let $\omega$ be a balanced metric on $X$ such that $\widetilde{\omega}^{n-1}=d\alpha$ for some $\widetilde{\omega}$-bounded ($2n-3$)-form $\alpha$ on $\widetilde{X}$, and let $\iota: H \xhookrightarrow{\hspace{14pt}}X$ be a complex hypersurface in $X$. It suffices to show that the image of $\iota_*:\pi_1(H)\longrightarrow\pi_1(X)$ is infinite. We argue by contradiction. Suppose that the image of $\iota_*:\pi_1(H)\longrightarrow \pi_1(X)$ is finite. Passing to a finite cover, one then obtains a compact hypersurface $\widetilde{H}\subset\widetilde{X}$ such that the restriction $\pi_{\lvert \widetilde{H}}:\widetilde{H}\longrightarrow H$ is finite (let us say of degree $k>0$). This would imply that:
    \begin{equation*}
        \text{Vol}_{\widetilde{\omega}}(\widetilde{H})=\dfrac{1}{(n-1)!}\int_{\widetilde{H}}\widetilde{\omega}^{n-1}=\dfrac{k}{(n-1)!}\int_H\omega^{n-1}=k\text{Vol}_\omega(H)>0
    \end{equation*}
    On the other hand, we have:
    \begin{equation*}
        \text{Vol}_{\widetilde{\omega}}(\widetilde{H})=\dfrac{1}{(n-1)!}\int_{\widetilde{H}}\widetilde{\omega}^{n-1}=\dfrac{1}{(n-1)!}\int_{\widetilde{H}}d\alpha=\dfrac{1}{(n-1)!}\int_{\partial\widetilde{H}}\alpha=0
    \end{equation*}
    due to the compactness of $\widetilde{H}$, yielding a contradiction. Hence, $\pi_1(X)$ is infinite.
\end{enumerate}
\end{proof}
\begin{rmk}
\begin{enumerate}
    \item Note that Observation~\ref{97}.(2) does not require the boundedness of $\alpha$, it only uses the $d$-exactness of $\widetilde{\omega}^{n-1}$.
    \item We can now infer that the degenerate balanced property is not stable under smooth modifications. Indeed, let $X$ be a compact degenerate balanced manifold, and consider $\mu:\widehat{X}\longrightarrow X$ to be the blow-up of $X$ at a point $x\in X$. Then the exceptional divisor $E:=\mu^{-1}(x)$ is a complex hypersurface in $\widehat{X}$, hence $\widehat{X}$ is \textbf{not} degenerate balanced due to Observation~\ref{97}.(1). Furthermore, $\widehat{X}$ is not even balanced hyperbolic since $E\simeq\CP^{n-1}$ and $\CP^{n-1}$ cannot be embedded in $X$. More precisely, we have:
\begin{prop}[\cite{marouani2023balanced}, Proposition 2.10] Let $X$ be a compact balanced hyperbolic manifold. Then there is no holomorphic map $f:\CP^{n-1}\longrightarrow X$ such that $f$ is non-degenerate at some point $x\in\CP^1$.
\end{prop}
\end{enumerate}
\end{rmk}
It is natural to ask about the converse: if $f:\widehat{X}\longrightarrow X$ is a smooth modification of compact complex manifolds with $\widehat{X}$ is degenerate balanced, then is $X$ degenerate balanced ? 

When dealing with smooth modifications of compact degenerate balanced manifolds, we adopt the following characterization:
\begin{prop}[\cite{popovici2015aeppli}, Proposition 5.4] A compact complex manifold $X$ is degenerate balanced if and only if it admits no non-zero $d$-closed $(1,1)$-current $T\geq 0$.
\label{95}
\end{prop}
The first part of Observation~\ref{97} is actually a direct consequence of Proposition~\ref{95}, since the existence of a complex hypersurface $H\subset X$ naturally gives rise to the \emph{current of integration} over $H$:
\begin{equation*}
     [H]:\mathscr{C}^\infty_{n-1,n-1}(X)\longrightarrow\C\quad,\qquad u\longmapsto\int_H u
\end{equation*}
This defines a non-zero $d$-closed positive ($1,1$)-current. 

Building on the ideas of \cite{popovici2013stability}, Theorem 2.2, we establish the following:
\begin{prop}
    Let $f:\widehat{X}\longrightarrow X$ be a smooth \textbf{surjective} modification of compact complex manifolds, with $\dim_\C X=n\geq 3$. If $\widehat{X}$ is degenerate balanced, then $X$ is again degenerate balanced. \label{96}
\end{prop}
\begin{proof} Let $E$ be the exceptional divisor of $f$ on $\widehat{X}$, and let $A\subset X$ be the analytic subset such that the restriction $f_{\lvert\widehat{X}\setminus E}:\widehat{X}\setminus E\longrightarrow X\setminus A$ is a biholomorphism. The proof proceeds by contradiction. If $X$ is not degenerate balanced, Proposition~\ref{95} yields the existence of a non-zero $d$-closed $(1,1)$-current $T\geq 0$. Then there exists an open cover $\{U_j\}_{j\in I}$ of $X$ such that $T_{\lvert_{U_j}}=i\partial\bar\partial\varphi_j$ for some plurisubharmonic function $\varphi_j:U_j\longrightarrow\R\cup\{-\infty\}$. In particular, $f^*T_{\lvert_{f^{-1}(u_j)}}=i\partial\bar\partial(\varphi_j\circ f)$, and these local pieces glue together into a globally defined $d$-closed ($1,1$)-current $f^*T\geq 0$ on $\widehat{X}$ thanks to the surjectivity of $f$ (see \cite{meo1996image} for further details). It remains to check that $f^*T$ is not the zero current on $\widehat{X}$ to conclude that $\widehat{X}$ is \textbf{not} degenerate balanced again thanks to Proposition~\ref{95}. Assume, for the sake of contradiction, that $f^*T=0$. This would mean that $\text{Supp\hspace{2pt}}T\subset A$. We distinguish the following two cases:

\textbf{Case 1.} If all the irreducible components $A_j$ of $A$ are such that $\text{codim}_\C A_j\geq 2$ in $X$, the support theorem [Corollary 2.11 in \cite{demailly1997complex}, Chapter III] forces $T=0$, which is a contradiction.

\textbf{Case 2.} If $A$ had certain global irreducible components $A_j$ with $\text{codim}_\C A_j=1$ in $X$, then $T=\underset{j}{\sum}\lambda_j[A_j]$ where $\lambda_j\geq 0$ and $[A_j]$ is the current of integration on $A_j$ thanks to another support theorem [Corollary 2.14 in \cite{demailly1997complex}, Chapter III]. So $f^*T=\underset{j}{\sum}\lambda_j[f^{-1}(A_j)]=0$ implies $\lambda_j=0$ ,$\forall j$, meaning that $T=0$ on $X$, yielding again a contradiction. Hence the desired result follows.
\end{proof}

\subsection{Deformations of Balanced hyperbolic manifolds}
The remainder of this section is devoted to the behavior of balanced hyperbolicity under small holomorphic deformations. Let ${X_t: t \in \B_\varepsilon(0)}$ be an analytic family of compact complex $n$-dimensional manifolds with $X_0$ balanced, where $\B_\varepsilon(0) \subset \C^N$ is the open ball of radius $\varepsilon>0$ centered at $0$. L. Alessandrini and G. Bassanelli observed that the balanced property is, in general, \textbf{not} preserved under small deformations; a classical counterexample is given by the small deformation of the Iwasawa manifold constructed by Nakamura (see \cite{nakamura1972complex}, \cite{alessandrini1990small}). A first extra condition ensuring stability is the $\partial\bar\partial$-property: C.-C. Wu proved in \cite{wu2006geometry} (Theorem 5.12) that satisfying the $\partial\bar\partial$-property is \textbf{open} under small deformations, and if $X_0$ is a balanced $\partial\bar\partial$-manifold, the nearby fibers $X_t$ are balanced $\partial\bar\partial$-manifolds. However, D. Popovici conjectured in \cite{popovici2017volume} that every compact $\partial\bar\partial$-manifold carries a balanced metric (Conjecture 6.1). Nevertheless, the $\partial\bar\partial$ assumption can be weakened while still preserving the deformation stability of the balanced property (see \cite{fu2011note}, \cite{ angella2017small}). One such condition will be discussed later in this section.

Let us consider the \textit{co}-\textit{polarized} case (This is a generalization of the polarization of a fiber by the De Rham class of a Kähler metric, which is possible thanks to Ehresmann’s Theorem \cite{ehresmann1947espaces}): Every holomorphic family of compact complex manifolds is locally $\mathscr{C}^\infty$ trivial. In particular, for every degree $k$ we have: $H^k_{DR}(X_t,\C)\simeq H^k_{DR}(X_0,\C)$). First, we recall the definition of \emph{Bott-Chern cohomology} and \emph{Aeppli cohomology}:
\begin{dfn} Let $X$ be a compact complex manifold with $\dim_\C X=n$.
\begin{enumerate}
    \item The \textbf{Bott-Chern cohomology group} of bidegree $(p,q)$ of $X$ is defined as follows:
    \begin{equation*}
        H_{BC}^{p,q}(X,\C)=\dfrac{\text{Ker}\left(\partial:\mathscr{C}^\infty_{p,q}(X)\longrightarrow\mathscr{C}^\infty_{p+1,q}(X)\right)\cap\text{Ker}\left(\bar\partial:\mathscr{C}^\infty_{p,q}(X)\longrightarrow\mathscr{C}^\infty_{p,q+1}(X)\right)}{\text{Im}\left(\partial\bar\partial:\mathscr{C}^\infty_{p-1,q-1}(X)\longrightarrow\mathscr{C}^\infty_{p,q}(X)\right)}
    \end{equation*}
    and the \textbf{Bott-Chern numbers} are $h^{p,q}_{BC}(X):=\dim_\C H^{p,q}_{BC}(X,\C)$.
    \item  The \textbf{Aeppli cohomology group} of bidegree $(p,q)$ of $X$ is defined as follows:
    \begin{equation*}
        H_A^{p,q}(X,\C)=\dfrac{\text{Ker}\left(\partial\bar\partial:\mathscr{C}^\infty_{p,q}(X)\longrightarrow\mathscr{C}^\infty_{p+1,q+1}(X)\right)}{\text{Im}\left(\partial:\mathscr{C}^\infty_{p-1,q}(X)\longrightarrow\mathscr{C}^\infty_{p,q}(X)\right)+\text{Im}\left(\bar\partial:\mathscr{C}^\infty_{p,q-1}(X)\longrightarrow\mathscr{C}^\infty_{p,q}(X)\right)}
    \end{equation*}
    and the \textbf{Aeppli numbers} are $h^{p,q}_A(X):=\dim_\C H^{p,q}_A(X,\C)$.
\end{enumerate}
\end{dfn}
We recall some of the key characterizations of $\partial\bar{\partial}$-manifolds, highlighting their good Hodge-theoretic properties (which is why they are sometimes referred to as \textbf{cohomologically Kähler}), following ideas from \cite{deligne1975real}, \cite{popovici2014deformation} and \cite{angella2013lemma} (see also the exposition in \cite{popovici2022non}, §1.3).
\begin{prop} Let $X$ be a compact complex manifold with $\dim_\C X=n$. The following statements are equivalent:
\begin{enumerate}
    \item $X$ is a $\partial\bar\partial$-manifold, i.e. $\forall p,q=0,\dots,n$, $\forall u\in\mathscr{C}_{p,q}^\infty(X)\cap\text{Ker }\partial\cap\text{Ker }\bar\partial$, we have:
    \begin{equation*}
        u\in\text{Im }d\Longleftrightarrow u\in\text{Im }\partial\Longleftrightarrow u\in\text{Im }\bar\partial\Longleftrightarrow u\in\text{Im }\partial\bar\partial
    \end{equation*}
    \item X have the \textbf{Hodge Decomposition property}, i.e. $\forall p,q=0,\dots,n$, every Dolbeault cohomology class $\{\alpha^{p,q}\}_{\bar\partial}\in H_{\bar\partial}^{p,q}(X,\C)$ can be represented by a $d$-closed $(p,q)$-form and for every $k=0,\dots,2n$, the linear map:
    \begin{equation*}
        \underset{p+q=k}{\bigoplus}H_{\bar\partial}^{p,q}(X,\C)\ni\underset{p+q=k}{\sum}\{\alpha^{p,q}\}_{\bar\partial}\longmapsto\left\{\underset{p+q=k}{\sum}\alpha^{p,q}\right\}\in H_{DR}^{p,q}(X,\C)
    \end{equation*}
    is \textbf{well-defined} by means of $d$-closed \textbf{pure-type} representatives $\alpha^{p,q}$ of their respective Dolbeault cohomology classes $($i.e. $\partial\alpha^{p,q}=\bar\partial\alpha^{p,q}=0)$ and \textbf{bijective}.
    \item $\forall p,q=0,\dots,n$, the canonical linear maps (induced by the identity):
    \begin{equation*}
        H_{BC}^{p,q}(X,\C)\longrightarrow H_{\bar\partial}^{p,q}(X,\C) \qquad \text{and}\qquad H_{\bar\partial}^{p,q}(X,\C)\longrightarrow H_A^{p,q}(X,\C)
    \end{equation*}
    are isomorphisms.
\end{enumerate} \label{94}
\end{prop}
\begin{dfn}[\cite{popovici2019holomorphic}, Definition 4.1] Let $\{X_t:t\in\B_\varepsilon(0)\}$ be an analytic family of compact $\partial\bar\partial$-manifolds of complex dimension $n$. Assume $X_0$ is balanced, and fix a balanced class $\{\omega_0^{n-1}\}_{BC}\in H_{BC}^{n-1,n-1}(X_0,\C)\subset H^{2n-2}_{DR}(X_0,\C)$. We say that a fiber $X_t$ is \textbf{co-polarized} by the De Rham class $\{\omega_0^{n-1}\}_{BC}$ if it is of type $(n-1,n-1)$ for the complex structure on $X_t$.
\end{dfn}
The deformation stability of balanced hyperbolicity on co-polarized fibers can be deduced from the following results:
\begin{thm}[\cite{popovici2019holomorphic}, Observation 7.2] Let $(X_0,\omega_0)$ be a balanced $\partial\bar\partial$-manifold. For $\varepsilon>0$ sufficiently small, the De Rham class $\{\omega_0^{n-1}\}_{BC}\in H^{2n-2}_{DR}(X_0,\C)$ defines a balanced class for the complex structure on the fibers $X_t$ that are co-polarized by it.
\end{thm}
\begin{lem}[\cite{chen2018compact}, Lemma 2.4] 
The $\widetilde{d}$-boundedness property is a \textbf{cohomological} property, i.e. if a closed $k$-form $\alpha$  on a compact Riemannian manifold $X$ is $\widetilde{d}($bounded$)$, every $k$-form $\beta\in\{\alpha\}\in H_{DR}^k(X,\R)$ is $\widetilde{d}($bounded$)$.
\end{lem}
Therefore, we arrive at the following conclusion:
\begin{cor} Let $\{X_t:t\in\B_\varepsilon(0)\}$ be an analytic family of compact $\partial\bar\partial$-manifolds with $\dim_\C X_t=n$, and assume $(X_0,\omega_0)$ balanced hyperbolic. Then the fibers $X_t$ co-polarized by $\{\omega_0^{n-1}\}\in H^{2n-2}_{DR}(X_0,\C)$ are also balanced hyperbolic $($for $t$ sufficiently close to $0)$.
\end{cor}

To establish a more general deformation result for balanced hyperbolic manifolds, we recall the following variants of the $\partial\bar\partial$-property in bidegree ($p,q$) (with $p,q\in\{0,\dots,n\}$) following Definition 5 in \cite{fu2011note}, Definition 2.1 in \cite{angella2017small}, Definition 3.1 and Theorem 1.1 in \cite{rao2019local}.
\begin{dfn} Let $X$ be a compact complex manifold with $\dim_\C X=n$, and let $p,q\in\{0,\dots,n\}$. $X$ is said to satisfy the :
\begin{enumerate}
    \item \textbf{weak $($p,q$)$-$\partial\bar\partial$-property} if for any form $\alpha\in\mathscr{C}^\infty_{p,q}(X)$ we have:
    \begin{equation*}
       \alpha\in\text{Im}\hspace{2pt}\partial\cap\text{Im}\hspace{2pt}\bar\partial \xLongrightarrow{\qquad} \bar\partial\alpha\in\text{Im}\hspace{2pt}\partial\bar\partial \vspace{-1ex}
    \end{equation*}
    \item \textbf{mild $($p,q$)$-$\partial\bar\partial$-property} if for any form $\alpha\in\mathscr{C}^\infty_{p,q}(X)$ we have:
    \begin{equation*}
        \alpha\in\text{Ker}\hspace{2pt}\bar\partial\cap\text{Im}\hspace{2pt}\partial \xLongrightarrow{\qquad} \alpha\in\text{Im}\hspace{2pt}\partial\bar\partial \vspace{-1ex}
    \end{equation*}
    \item \textbf{strong $($p,q$)$-$\partial\bar\partial$-property} if for any form $\alpha\in\mathscr{C}^\infty_{p,q}(X)$ we have:
    \begin{equation*}
        \alpha\in\text{Ker}d\cap(\text{Im}\hspace{2pt}\partial+\text{Im}\hspace{2pt}\bar\partial) \xLongrightarrow{\qquad} \alpha\in\text{Im}\hspace{2pt}\partial\bar\partial \vspace{-1ex}
    \end{equation*}
\end{enumerate} \label{90}
\end{dfn}
\begin{rmk}
\begin{enumerate}
    \item For any bidegree ($p,q$), we have : 
    \begin{align*}
        \text{The strong} (p,q)\text{-}\partial\bar\partial\text{-property}&\xLongrightarrow{\qquad}\text{The mild} (p,q)\text{-}\partial\bar\partial\text{-property} \\ &\xLongrightarrow{\qquad}\text{The weak} (p,q)\text{-}\partial\bar\partial\text{-property}
    \end{align*}
    \item While the \emph{strong} ($p,q$)-$\partial\bar\partial$-property are stable under small deformations of the complex structure (Proposition 4.8, \cite{angella2017small}), the \emph{weak} ($p,q$)-$\partial\bar\partial$-property and the \emph{mild} ($p,q$)-$\partial\bar\partial$-property are \textbf{not} open under deformations (see \cite{ugarte2015balanced}, Example 3.7 for the case ($n-1,n$)).
\end{enumerate}
\end{rmk}
If the weak ($n-1,n$)-$\partial\bar\partial$-property holds on the fibers $X_t$, for $t$ sufficiently close to $0$, then the balanced property is open under small deformations (\cite{fu2011note}, Theorem 6). Hence, a first result in the desired direction is the following:
\begin{thm}
Let $\{X_t:t\in\B_\varepsilon(0)\}$ be an analytic family of compact complex $n$-dimensional manifolds such that $X_0$ is balanced. Assume that the fibers $X_t$ satisfy the weak $(n-1,n)$-$\partial\bar\partial$-property for $t\neq 0$.
\begin{enumerate}
    \item If $X_0$ is balanced hyperbolic, the fibers $X_t$ carry balanced metrics $\omega_t$ that are \textbf{Gauduchon hyperbolic} for small $t$.
    \item If moreover, $X_0$ is degenerate balanced and the fibers $X_t$ satisy the mild $(n,n-2)$-$\partial\bar\partial$-property for $t\neq 0$, the fibers $X_t$ are degenerate balanced for small $t$.
\end{enumerate}  \label{71}
\end{thm}
Here, \emph{Gauduchon hyperbolicity} is understood in the sense of Definition 3.8 of \cite{marouani2023skt} (its precise definition will be recalled in the next section; Definition~\ref{99}).
\begin{proof} Let us first revisit the beginning of the proof of Theorem 6 in \cite{fu2011note}.
 Let $X$ be the underlying differentiable manifold of the given analytic family, and let $\omega$ be a balanced metric on $X_0$, and let $\{h_t\}_{t\in\B_\varepsilon(0)}$ be a $\mathscr{C}^\infty$-family of Hermitian metrics on $(X_t)_{t\in\B_\varepsilon(0)}$ with $h_0=\omega$. We decompose $\Omega:=\omega^{n-1}$ with respect to the complex structure on $X_t$:
\begin{equation}
    \Omega=\Omega_t^{n,n-2}+\Omega_t^{n-1,n-1}+\Omega_t^{n-2,n}\hspace{2pt},\qquad t\in\B_\varepsilon(0) \label{69}
\end{equation}
then $\Omega_t^{n-1,n-1}>0$ for $t$ close enough to 0 by continuity (since $\Omega_0^{n-1,n-1}=\omega^{n-1}>0$). Since $\Omega$ is $d$-closed, we get: $\bar\partial_t\Omega_t^{n-1,n-1}+\partial_t\Omega_t^{n-2,n}=0$, meaning that $\bar\partial_t\Omega_t^{n-1,n-1}\in\mathscr{C}^\infty_{n-1,n}(X_t)$ lies in $\text{Im}\hspace{2pt}\partial_t\cap\text{Im}\hspace{2pt}\bar\partial_t$. Applying the weak ($n-1,n$)-$\partial\bar\partial$-property yields: $\bar\partial_t\Omega_t^{n-1,n-1}=i\partial_t\bar\partial_tu_t$ for some $u_t\in\mathscr{C}^\infty_{n-2,n-1}(X_t)\cap(\text{Ker}(\partial_t\bar\partial_t))^\perp$. Then:
\begin{equation}
    \Omega_t:=\Omega_t^{n-1,n-1}+i\partial_tu_t-i\bar\partial_t\bar u_t\hspace{2pt},\qquad \text{ for }\hspace{2pt}t\sim 0 \label{72}
\end{equation}
defines a continuous family of positive ($n-1,n-1$)-forms on $X_t$ for $t$ close enough to 0 (see the proof of Theorem 6, \cite{fu2011note}). Hence, there exist a balanced metric $\omega_t$ on each respective fiber $X_t$ for $t$ sufficiently small, defined as the ($n-1$)-th root of $\Omega_t$ (Thanks to Michelson's trick (ii)', p.19 in \cite{michelsohn1982existence}), with $\omega_0=\omega$. 
\begin{enumerate}
    \item Suppose that $\omega$ is balanced hyperbolic. Then $\widetilde{\Omega}=\widetilde{\omega}^{n-1}=d\Gamma$ for some smooth and $\widetilde{\omega}$-bounded ($2n-3$)-form $\Gamma$ on $\widetilde{X}$. Then by identifying the bidegree ($n-1,n-1$)-component of this expression with respect to the complex structure on $\widetilde{X}_t$, we get:
    \begin{equation}
        \widetilde{\Omega_t}^{n-1,n-1}=\partial_t\Gamma_t^{n-2,n-1}+\bar\partial_t\Gamma_t^{n-1,n-2}\hspace{2pt},\qquad\forall t\in\B_\varepsilon(0) \label{70}
    \end{equation}
    After lifting~(\ref{72}) to $\widetilde{X}_t$, (\ref{70}) implies:
    \begin{equation}
        \widetilde{\Omega_t}=\partial_t(\Gamma_t^{n-2,n-1}+i\pi^*u_t)+\bar\partial_t(\Gamma_t^{n-1,n-2}-i\pi^*\bar u_t)\hspace{2pt},\qquad\forall t\in\B_\varepsilon(0)
    \end{equation}
    Put: $\alpha_t:=\Gamma_t^{n-2,n-1}+i\pi^*u_t$. Since $\widetilde{\Omega_t}^{n-1,n-1}$ is real, one can choose $\Gamma$ such that $\Gamma_t^{n-1,n-2}=\overline{\Gamma_t^{n-2,n-1}}$. Then:
    \begin{equation}
       \widetilde{\Omega_t}=\widetilde{\omega_t}^{n-1}=\partial_t\alpha_t+\bar\partial_t\bar\alpha_t\hspace{2pt},\qquad\text{ for }\hspace{2pt}t\sim 0 
    \end{equation}
    with:
    \begin{equation*}
    \lVert\alpha_t\rVert_{L^\infty_{\widetilde{h_t}}(\widetilde{X}_t)}\leq \lVert\Gamma_t^{n-2,n-1}\rVert_{L^\infty_{\widetilde{h_t}}(\widetilde{X}_t)}+\lVert\pi^*u_t\rVert_{L^\infty_{\widetilde{h_t}}(\widetilde{X}_t)}\leq C_t\lVert\Gamma\rVert_{L^\infty_{\widetilde{\omega}}(\widetilde{X})}+\lVert u_t\rVert_{L^\infty_{h_t}(X_t)}<\infty
    \end{equation*}
    where the constant $C_t>0$ depends only on the metric $h_t$ on the \textbf{compact} fiber $X_t$ (the metrics $\omega$ and $h_t$ are comparable on $X$, $\forall t\sim 0$). Hence, the balanced metrics $\omega_t$ are Gauduchon hyperbolic on the respective fibers $X_t$ for $t\sim 0$.
    \item Now let $\omega$ be a degenerate balanced metric on $X_0$, that is $\Omega:=\omega^{n-1}=d\alpha$ for some $\alpha\in\mathscr{C}_{2n-3}^\infty(X)$ (in particular $n\geq3$). Then the decomposition (\ref{69}) yields:
\begin{equation}
    \Omega_t^{n-1,n-1}=\bar\partial_t\alpha_t^{n-1,n-2}+\partial_t\alpha_t^{n-2,n-1}\hspace{2pt},\qquad t\in\B_\varepsilon(0)
\end{equation}
Since $\Omega_t^{n-1,n-1}$ is real, one can choose $\alpha$ such that $\alpha_t^{n-1,n-2}=\overline{\alpha_t^{n-2,n-1}}$. Hence, one gets:
\begin{equation}
    \Omega_t=\partial_t(\alpha_t^{n-2,n-1}+iu_t)+\bar\partial_t(\alpha_t^{n-1,n-2}-i\bar u_t)=\bar v_t+v_t\hspace{2pt},\qquad t\in\B_\varepsilon(0)
\end{equation}
where $v_t:=\bar\partial_t(\alpha_t^{n-1,n-2}-i\bar u_t)\in\mathscr{C}^\infty_{n-1,n-1}(X_t)\overset{(*)}{\cap}\text{Ker}\hspace{2pt}\partial_t{\cap}\text{Im}\hspace{2pt}\bar\partial_t$ (here ($*$) follows from the fact that $\partial_t\Omega_t=\bar\partial_t\Omega=0$). It remains to check that the following system is solvable for: 
\begin{equation}
\begin{cases}
\bar\partial w_t= v_t \quad,\qquad t\neq0\\
\partial w_t=0
\end{cases} \label{67}
\end{equation}
This would imply that $\omega_t^{n-1}=\Omega_t=\bar v_t+v_t=\partial_t\bar w_t+\bar\partial_tw_t=d(w_t+\bar w_t)$, hence the metrics $\omega_t$ are degenerate balanced. 

Since $v_t\in\text{Im}\bar\partial_t$, there exists $\beta_t\in\mathscr{C}^\infty_{n-1,n-2}(X_t)$, such that $v_t=\bar\partial_t\beta_t$, and since $\partial_tv_t=0$, we have $\partial_t\beta_t\in\text{Ker}\hspace{2pt}\bar\partial_t\cap\text{Im}\hspace{2pt}\partial_t$. By applying the mild ($n,n-2$)-$\partial\bar\partial$-property on each fiber $X_t$, we get $\partial_t\beta_t=\partial_t\bar\partial_t\gamma_t$, for some $\gamma_t\in\mathscr{C}^\infty_{n-1,n-3}(X_t)$. Then $w_t:=\beta_t-\bar\partial_t\gamma_t$ satisfies the system~(\ref{67}). Hence the desired result follows.
\end{enumerate}
\end{proof}
\begin{rmk}
We propose the following "alternative version" of Theorem~\ref{71} in the case where the family satisfies the $\partial\bar\partial$-property.
\end{rmk}
\begin{thm} Let $\{X_t:t\in\B_\varepsilon(0)\}$ be an analytic family of compact $\partial\bar\partial$-manifolds of complex dimension $n$ such that $X_0$ is balanced.
\begin{enumerate}
    \item Assume $X_0$ is balanced hyperbolic, and let $\omega$ be a balanced hyperbolic metric on $X_0$. Then $\omega$ extends to a $\mathscr{C}^\infty$-family of balanced metrics $\{\omega_t\}_{t\sim 0}$ that are Gauduchon hyperbolic on the respective fibers $X_t$ for small $t$.
    \item If moreover, $\omega$ is degenerate balanced, the $\mathscr{C}^\infty$-family of balanced metrics $\{\omega\}_{t\sim0}$ on the fibers $\{X_t\}_{t\sim 0}$ are degenerate balanced with a $\mathscr{C}^\infty$-family of $d$-potentials $\{w_t\}_{t\sim 0}$ for small $t$, i.e. $\omega=dw_0$ and $\omega_t=dw_t$, $\forall t\sim 0$.
\end{enumerate}  \label{68}
\end{thm}
\begin{proof}
    \begin{enumerate}
        \item As in Theorem~\ref{71}, let $\omega$ be a balanced hyperbolic metric on $X_0$, and let $\{h_t\}_{t\in\B_\varepsilon(0)}$ be a $\mathscr{C}^\infty$-family of Hermitian metrics on $(X_t)_{t\in\B_\varepsilon(0)}$ with $h_0=\omega$, and let $\Omega:=\omega^{n-1}$. When decomposing $\Omega$ with respect to the complex structure of a given fiber $X_t$, we get thanks to the $\partial_t\bar\partial_t$-property that: $\bar\partial_t\Omega_t^{n-1,n-1}$ is $\partial_t\bar\partial_t$-exact, i.e. $\bar\partial_t\Omega_t^{n-1,n-1}=i\partial_t\bar\partial_tu_t$ for some $u_t\in\mathscr{C}^\infty_{n-2,n-1}(X_t)\cap(\text{Ker}(\partial_t\bar\partial_t))^\perp$. The form $u_t$ can be chosen as the minimal $L^2_{h_t}$-norm solution of the equation $\bar\partial_t\Omega_t^{n-1,n-1}=i\partial_t\bar\partial_tu_t$, which is given by the Neumann-type formula (See \cite{popovici2015aeppli}, Theorem 4.1):
        \begin{equation*}
         u_t=(\partial_t\bar\partial_t)^*\G_{BC,t}(\bar\partial_t(i\Omega_t^{n-1,n-1}))\quad,\qquad t\sim 0
        \end{equation*}
         where $\G_{BC}$ is the \textbf{Green operator} associated with the \textbf{Bott-Chern Laplacian} $\Delta_{BC}$ defined as follows (see \cite{kodaira1960deformations}, Section $6$):
         \begin{equation*}
         \Delta_{BC}=\partial^*\partial+\bar\partial^*\bar\partial+(\partial\bar\partial)^*(\partial\bar\partial)+(\partial\bar\partial)(\partial\bar\partial)^*+(\partial^*\bar\partial)^*(\partial^*\bar\partial)+(\partial^*\bar\partial)(\partial^*\bar\partial)^*.
         \end{equation*}
         Then $\{u_t\}_{t\sim0}$ defines a $\mathscr{C}^\infty$-family of ($n-2,n-1$)-forms on the fibers $(X_t)_{t\sim 0}$ thanks to Theorem 5 in \cite{kodaira1960deformations}. Hence the family $\{\Omega_t\}_{t\in\B_\varepsilon(0)}$ defined by the formula~(\ref{72}) gives a $\mathscr{C}^\infty$-family of positive ($n-1,n-1$)-forms on the fibers $\{X_t\}_{t\sim0}$ for $t$ small enough, with $\Omega_0=\Omega$. Thus, taking the ($n-1)$-root oe each $\Omega_t$ defines a $\mathscr{C}^\infty$-family of balanced metrics $\{\omega_t\}_{t\sim0}$ on the respective fibers $X_t$, that are Gauduchon hyperbolic according to Theorem~\ref{71}.(1) with $\omega_0=\omega$.
         \item Moreover, if $\omega$ is degenerate balanced, the family of forms $w_t$ defined by the system (\ref{67}) (which is solvable thanks to the $\partial_t\bar\partial_t$-property) can be taken as the minimal $L^2_{h_t}$-norm solutions given by the following formula (Lemma 2.4, \cite{dinew2021generalised}):
         \begin{equation*}
         w_t=\G_{BC,t}(\bar\partial_t^*v_t+\bar\partial_t^*\partial_t\partial_t^*v_t)
         \end{equation*}
         Hence the family $\{w_t\}_{t\sim 0}$ varies in a smooth way with respect to $t$ thanks to Theorem 5 in \cite{kodaira1960deformations}, and the $\mathscr{C}^\infty$-family of ($n-1,n-1$) forms $\{\omega_t^{n-1}\}_{t\sim 0}$ admits a $\mathscr{C}^\infty$-family of $d$-potentials.
    \end{enumerate}
\end{proof}
A consequence of Theorem~\ref{71} is the following:
\begin{cor}
If $\{X_t:t\in\B_\varepsilon(0)\}$ is an analytic family of compact complex $n$-dimensional manifolds such that $X_0$ is balanced hyperbolic with $H_{\bar\partial}^{2,0}(X_0,\C)=0$, the fibers $X_t$ are Gauduchon hyperbolic for small $t$. \label{91}
\end{cor}
\begin{proof}
    Assume $X_0$ is a balanced hyperbolic manifold and $H_{\bar\partial}^{2,0}(X_0,\C)=0$. Since the function $t\in\B_\varepsilon(0)\longmapsto h_{\bar\partial_t}^{p,q}(t)=\dim_\C H_{\bar\partial_t}^{p,q}(X_t,\C)$ is upper semicontinuous in $t$ (Theorem 6, \cite{kodaira1960deformations}), we get $h_{\bar\partial_t}^{2,0}(X_t)\leq h_{\bar\partial}^{2,0}(X_0)=0$, which implies that $H_{\bar\partial_t}^{2,0}(X_t,\C)=0$ for small $t$. By Serre duality, we deduce that $H_{\bar\partial_t}^{n-2,n}(X_t,\C)=0$. Now, let $\alpha_t\in\mathscr{C}^\infty_{n-1,n-1}(X_t)$ be such that $\bar\partial_t\alpha_t\in\text{Im}\partial_t$, i.e. $\bar\partial_t\alpha_t=\partial_t\beta_t$, for some $\beta_t\in\mathscr{C}^\infty_{n-2,n}(X_t)$. Since $\bar\partial_t\beta_t=0$ (for bidegree reasons) and $H_{\bar\partial_t}^{n-2,n}(X_t,\C)=0$, $\beta_t\in\text{Im}\bar\partial_t$. Hence $\bar\partial_t\alpha_t\in\text{Im}(\partial_t\bar\partial_t)$, which means that the ($n-1,n$)-th weak $\partial\bar\partial$-Lemma is satisfied on $X_t$. We conclude that the nearby fibers $X_t$ are also balanced hyperbolic (for $t\sim 0$) thanks to Theorem~\ref{71}.
\end{proof}
\begin{ex}
\begin{enumerate}
    \item Under mild hypotheses, R.~Friedman proved that the connected sums $X=\#_k(S^3\times S^3)$ for $k\geq 2$ are $\partial\bar{\partial}$-manifolds (viewed as small smoothings of \textbf{Clemens manifolds}) \cite{friedman2020lemma}. This result was later extended to full generality by C.~Li \cite{li2024polarized}, thereby answering positively Question~5.6 in~\cite{popovici2015aeppli}. It follows from Theorem \ref{68}.(2) that any degenerate balanced metric $\omega$ on $X$ extends to a $\mathscr{C}^\infty$-family of degenerate balanced metrics $\{\omega\}_{t\sim 0}$ on any small family $\{X_t\}_{t\sim 0}$ of deformations of $X$.
    \item Set $X=\#_k(\Sp^3\times\Sp^3)$, and let $Y$ be a \emph{fake projective plane} (or a \emph{Mumford surface}); this is a compact complex surface with the same cohomology as $\CP^2$ (see \cite{mumford1979algebraic}). It can be viewed as an arithmetic quotient of the hyperbolic ball $\B^2_\C$. Since $X$ is a degenerate balanced $\partial\bar\partial$-manifold and $Y$ is Kähler hyperbolic, then $Z:=X\times Y$ is a balanced hyperbolic $\partial\bar\partial$-manifold. Applying the Künneth formula and the Hodge decomposition yields :
    \begin{equation*}
        h^{2,0}(Z)=h^{2,0}(X)+h^{1,0}(X)h^{1,0}(Y)+h^{2,0}(Y)=0
    \end{equation*}
    Hence any small deformation of $Z$ is balanced Gauduchon hyperbolic thanks to Corollary \ref{91}. Moreover, Theorem \ref{68}.(1) yields that given any balanced hyperbolic metric on $Z$, one construct a $\mathscr{C}^\infty$-family of balanced metrics $\{\omega_t\}_t$ that are Gauduchon hyperbolic on small deformations $\{Z_t\}_t$ of $Z$.
\end{enumerate}
\end{ex}
In the spirit of Conjecture 4.25 in \cite{khelifati2025holomorphic}, we propose the following conjecture:
\begin{conj}
    Let $X$ be a compact balanced manifold. If $X$ admits a balanced metric $\omega$ which is Gauduchon hyperbolic, then $X$ is balanced hyperbolic. \label{98}
\end{conj}

In order to address the deformation openness of balanced hyperbolicity and answer positively Conjecture~\ref{98}, one would need to solve the $\bar\partial$ and $\partial$-problems of the form \ref{67} and \ref{69} in bidegrees $(n-2,n-1)$ and $(n-1,n-2)$, respectively, with $L^\infty$-estimates on a complete balanced manifold (equipped with a balanced metric $\omega$ such that $\omega^{n-1}$ is ($\partial+\bar\partial$)(bounded) if needed). At present, this appears to be out of reach, even in the compact setting.

Even obtaining an analogue of Corollary 4.20 in \cite{khelifati2025holomorphic} appears to be highly nontrivial. Indeed, to the best of our knowledge, no satisfactory bounds are currently available for the spectra of the Laplacians in bidegrees $(n-1,n-2)$ and $(n-2,n-1)$, in contrast with the situation in bidegrees $(n,n-1)$ and $(n-1,n)$, where such estimates have been established in \cite{marouani2022some}, Corollary 3.7. This difficulty is further compounded by the fact that, on a balanced manifold, the Bott-Chern Laplacian $\Delta_{BC}$ does not commute with the $\partial$ and $\bar\partial$ operators and their adjoints $\partial^*$ and $\bar\partial^*$, while the commutation deficit remains hard to control.

\section{degree \texorpdfstring{$2p$}{2p} Hyperbolicity Notions}
Let us recall some of the basic definitions. Let $X$ be a compact complex manifold with $\dim_\C X=n$, and $\pi:\widetilde{X}\longrightarrow X$ be its universal cover. Recall that for any $k$-form $\alpha$ in $X$, we denote its pullback to the universal cover $\widetilde{X}$ by $\pi^*\alpha=:\widetilde{\alpha}$.
\begin{dfn}[\textbf{Special metrics}] A Hermitian metric $\omega$ on $X$ is said to be:
\begin{enumerate}
    \item\cite{gauduchon19841} \textbf{Strong Kähler with torsion} $($\textbf{SKT} for short$)$ or \textbf{pluriclosed} if $\partial\bar\partial\omega=0$.
    \item \cite{gauduchon1977theoreme} \textbf{Gauduchon} $($or \textbf{standard}$)$ if $\partial\bar\partial\omega^{n-1}=0$.
    \item $($\cite{sullivan1976cycles},\cite{harvey1983intrinsic},\cite{gauduchon19841}, \cite{streets2010parabolic}$)$ \textbf{Hermitian-symplectic} $($\textbf{HS} for short$)$ if $\exists\alpha\in\mathscr{C}^\infty_{2,0}(X)$ such that $d(\alpha+\omega+\bar\alpha)=0$. Or equivalently, $\partial\omega$ is $\bar\partial$-exact.
\end{enumerate}
\end{dfn}
We also recall the definition of the hyperbolicity notions associated with SKT and Gauduchon metrics on a compact complex manifold $X$.
\begin{dfn}[\textbf{SKT and Gauduchon hyperbolicity}, \cite{marouani2023skt}] 
\begin{enumerate}
    \item A $(p,q)$-form $\eta$ on $(X,\omega)$ is said to be \textbf{$(\widetilde{\partial+\bar\partial})($bounded$)$} if there exist a $\widetilde{\omega}$-bounded $(p-1,q)$-form $\alpha$ and a $\widetilde{\omega}$-bounded $(p,q-1)$-form $\beta$ on $\widetilde{X}$ such that $\widetilde{\eta}=\partial\alpha+\bar\partial\beta$.
    \item $X$ is said to be \textbf{SKT hyperbolic} if it admits a $(\widetilde{\partial+\bar\partial})($bounded$)$ SKT metric.
    \item $X$ is said to be \textbf{Gauduchon hyperbolic} if it admits a Gauduchon metric $\omega$ such that $\omega^{n-1}$ is $(\widetilde{\partial+\bar\partial})($bounded$)$.
\end{enumerate} \label{99}
\end{dfn}
\subsection{Hermitian-symplectic Hyperbolicity}
First, we recall the definition of \emph{sG} metrics and \emph{sG hyperbolicity}:
\begin{dfn}[\cite{popovici2013deformation}, Definition 4.1 \& Proposition 4.2] Let $X$ be a compact complex manifold with $\dim_\C X=n$. A Hermitian metric $\omega$ on $X$ is \textbf{strongly Gauduchon} $($or \textbf{sG} for short$)$ if one of the following equivalent conditions holds:
\begin{enumerate}
    \item The $(n,n-1)$-form $\partial\omega^{n-1}$ is $\bar\partial$-exact.
    \item $\exists\alpha\in\mathscr{C}^\infty_{n,n-2}(X)$ such that $d(\alpha+\omega^{n-1}+\bar\alpha)=0$.
\end{enumerate}
If $X$ carries such a metric, we say that $X$ is a \textbf{strongly Gauduchon} $($or \textbf{sG}$)$ manifold.
\end{dfn}
\begin{dfn}[\cite{ma2024strongly}, Definition 2.2] We say that $X$ is \textbf{sG hyperbolic} if there exists an sG metric $\omega$ on $X$ such that $\omega^{n-1}$ is the $(n-1,n-1)$-component of a real $\widetilde{d}($bounded$)$, $d$-closed smooth $(2n-2)$-form $\Omega$.
\end{dfn}
Now, we define the notion of \textit{Hermitian-symplectic hyperbolicity} (or \textit{HS hyperbolicity} for short) on compact complex manifolds:
\begin{dfn}
We say that a compact complex manifold $X$ with $\dim_\C X=n$ is \textbf{HS hyperbolic} if it admits a real smooth closed $2$-form $\Omega$ such that:
\begin{enumerate}
    \item $\omega:=\Omega^{1,1}$ is positive definite, i.e. $\omega$ is a Hermitian metric.
    \item $\Omega$ is $\widetilde{d}($bounded$)$, i.e. $\widetilde{\Omega}=d\eta$ for some \textbf{bounded} $1$-form $\eta$ on $\widetilde{X}$.
\end{enumerate}
We also say that the metric $\omega$ is \textbf{HS hyperbolic}.
\end{dfn}
After all these definitions of hyperbolicity notions, we observe that the hierarchy that we have on special metrics remains preserved in the hyperbolic setting, that is to say:
\begin{prop} Let $(X,\omega)$ be a compact Hermitian manifold with $\dim_\C X=n$. Then:
\begin{equation*}
\begin{tikzcd}
\text{$\omega$ is Kähler hyperbolic} \arrow[Rightarrow]{r}{(1)} \arrow[Rightarrow]{d}[swap]{(3)} & \text{$\omega$ is HS hyperbolic} \arrow[Rightarrow]{r}{(2)} & \text{$\omega$ is SKT hyperbolic} \\
\text{$\omega$ is balanced hyperbolic} \arrow[Rightarrow]{r}[swap]{(4)} & \text{$\omega$ is sG hyperbolic} \arrow[Rightarrow]{r}{(5)} & \text{$\omega$ is Gauduchon hyperbolic}
\end{tikzcd} 
\end{equation*} \label{75} \vspace{-3ex}
\end{prop}
Almost all of the implications have been proved in the previously mentioned articles (where these hyperbolicity notions were defined). We provide proofs of the implications in this diagram to make the exposition self-contained.
\begin{proof}
\begin{enumerate}[label=(\arabic*)]
    \item If $\omega$ is Kähler hyperbolic, $\Omega=\omega$ is a smooth real d-closed $2$-form which is $\widetilde{d}$(bounded) and $\Omega^{1,1}=\Omega=\omega$.
    \item $\omega$ being HS hyperbolic means that $\omega>0$ and there exists a smooth $(2,0)$-form $\alpha$ in $X$ such that $\Omega=\alpha+\omega+\bar\alpha$ is a $d$-closed, $\widetilde{d}$(bounded) $2$-form, i.e. $\widetilde{\Omega}=d\eta$, where $\eta$ is a $\widetilde{\omega}$-bounded $1$-form on $\widetilde{X}$. This implies that: \begin{enumerate}
        \item $\partial\bar\partial\omega=0$ since $d\Omega=0$ $\implies$ $\omega$ is SKT.
        \item $\widetilde{\omega}=\partial\eta^{0,1}+\bar\partial\eta^{1,0}$ (by taking the $(1,1)$-part in $\widetilde{\Omega}=d\eta$). Hence $\omega$ is SKT hyperbolic.
    \end{enumerate}
    \item If $\omega$ is Kähler hyperbolic, that is $\widetilde{\omega}=d\eta$ for some $\widetilde{\omega}$-bounded $1$-form $\eta$ on $\widetilde{X}$, then:
    \begin{equation*}
        \widetilde{\omega}^{n-1}=\widetilde{\omega}^{n-2}\wedge\widetilde{\omega}=\widetilde{\omega}^{n-2}\wedge d\eta=d(\widetilde{\omega}^{n-2}\wedge \eta) \qquad \text{since $\widetilde{\omega}$ is closed}
    \end{equation*}
    Since $\widetilde{\omega}^{n-2}\wedge \eta$ is bounded, $\omega$ is balanced hyperbolic.
    \item This is obvious (the proof is similar to (1)).
    \item The proof is similar to (2).
\end{enumerate}  
\end{proof}
S. Marouani proved in \cite{marouani2023skt} that SKT hyperbolicity implies Kobayashi hyperbolicity (Theorem 3.5). Thanks to (2) in Proposition \ref{75}, we deduce that:
\begin{cor}
Every HS hyperbolic compact complex manifold is Kobayashi hyperbolic. \label{83}
\end{cor}
Based on the fact that every compact Hermitian-symplectic manifold is strongly Gauduchon (as shown in \cite{yau2023strominger}, Lemma 1, §.2, and further detailed in \cite{dinew2021generalised}, Proposition 2.1), one can show that the following holds:
\begin{prop}
Every compact HS hyperbolic manifold $X$ is sG hyperbolic. \label{76}
\end{prop}
\begin{proof}
Let $\omega$ be a HS hyperbolic metric, this means that $\omega>0$ and $\exists\alpha\in\mathscr{C}^\infty_{2,0}(X)$ such that $\Omega=\alpha+\omega+\bar\alpha$ is $d$-closed and $\widetilde{\Omega}=d\eta$ for some $\widetilde{\omega}$-bounded $1$-form $\eta$ on $\widetilde{X}$. Put $\Gamma:=\Omega^{n-1}$, then:
\begin{enumerate}
    \item $\widetilde{\Gamma}=\widetilde{\Omega}^{n-2}\wedge d\eta=d(\widetilde{\Omega}^{n-2}\wedge \eta)$\quad (since $\widetilde{\Omega}$ is closed) $\xLongrightarrow{\qquad}$ $\Gamma$ is $\widetilde{d}$(bounded).
    \item $\displaystyle\Gamma=(\alpha+\omega+\bar\alpha)^{n-1}=\underset{k=0}{\overset{n-1}{\sum}}\underset{l=0}{\overset{k}{\sum}}\binom{n-1}{k}\binom{k}{l}\alpha^l\wedge\bar\alpha^{k-l}\wedge\omega^{n-k-1}$. By taking the bidegree $(n-1,n-1)$ in this expression (which corresponds to $l=k-l$), we get:
    \begin{align*}
        \Gamma^{n-1,n-1}&=\underset{l=0}{\overset{\left\lfloor\frac{n-1}{2}\right\rfloor}{\sum}}\binom{n-1}{2l}\binom{2l}{l}\alpha^l\wedge\bar\alpha^l\wedge\omega^{n-2l-1} \\
        &=\omega^{n-1}+2\binom{n-1}{2}\alpha\wedge\bar\alpha\wedge\omega^{n-3}+6\binom{n-1}{4}\alpha^2\wedge\bar\alpha^2\wedge\omega^{n-5}+\cdots \\ &\overset{(*)}{\geq}\omega^{n-1}>0
    \end{align*}
    where the inequality $(*)$ comes from the fact that $\alpha^l\wedge\bar\alpha^l\wedge\omega^{n-2l-1}\geq 0$ (see Chapter III in \cite{demailly1997complex}, Example 1.2 and Proposition 1.11). We conclude thanks to Michelsohn's trick that $\omega':=\sqrt[n-1]{\Gamma^{n-1,n-1}}$ defines a sG hyperbolic metric on $X$.
\end{enumerate}
\end{proof}
\subsection{\texorpdfstring{$p$}{p}-SKT and \texorpdfstring{$p$}{p}-HS Hyperbolicity}
Now, we will generalize the (first-order) notions of SKT hyperbolicity and HS hyperbolicity to notions of hyperbolicity of degree $2p$ for $2\leq p\leq n-1$, as was done in \cite{haggui2023compact}, where Gromov’s Kähler hyperbolicity was generalized to \textit{p-Kähler hyperbolicity}. But first, let us recall the definition of a higher degree positivity notion (see Definition 1.1 in \cite{harvey1972positive}, and Definition 1.1, Chapter III in \cite{demailly1997complex}). This notion has been referred to as \textbf{transversality} in \cite{sullivan1976cycles}, Definition I.3, and in \cite{alessandrini1987closed}, Definition 1.3.
\begin{dfn} Let $E$ be a $\C$-vector space with $\dim_\C E=n$, and $E^*$ its dual. And let $1\leq p,q\leq n-1$ be such that $p+q=n$.
\begin{enumerate}[label=(\arabic*)]
    \item A $(p,p)$-form $u\in\Lambda^{p,p}(E^*)$ is \textbf{weakly strictly positive} if for any linearly independent $(1,0)$-forms $\alpha_1,\dots,\alpha_q\in E^*$ we have:
    \begin{equation*}
        u\wedge i\alpha_1\wedge\bar\alpha_1\wedge\cdots\wedge i\alpha_q\wedge\bar\alpha_q
    \end{equation*}
    is a strictly positive volume form.
    \item Let $X$ be a compact complex manifold with $\dim_\C X=n$. A $(p,p)$-form $\alpha\in\mathscr{C}^\infty_{p,p}(X)$ is \textbf{weakly strictly positive} if $\forall x\in X$, $\alpha(x)\in\Lambda^{p,p}T_x^*X$ is weakly strictly positive in the sense of $(1)$.
\end{enumerate}
\begin{rmk}[Corollary 1.5-Chapter III, \cite{demailly1997complex}]
All weakly strictly positive $(p,p)$-forms $u$ are \textbf{real}, i.e. $\bar u=u$. 
\end{rmk}
\end{dfn}
Here, we recall the generalizations of special metrics in degree $2p$, with $1\leq p\leq n-1$.
\begin{dfn} Let $X$ be a compact complex manifold with $\dim_\C X=n\geq 2$, and fix $1\leq p\leq n-1$. Then $X$ is said to be:
\begin{enumerate}
    \item \textbf{p-Kähler} if it admits a $d$-closed weakly strictly positive $(p,p)$-form $\Omega$. Here, $\Omega$ is called a \textbf{p-Kähler form} $($Definition 1.11, \cite{alessandrini1987closed}$)$.
    \item \textbf{p-Hermitian-symplectic} $($or \textbf{p-HS} for short$)$ if it admits a real $d$-closed $2p$-form $\Omega$ such that $\Omega^{p,p}$ is weakly strictly positive. Here, $\Omega^{p,p}$ is called a \textbf{p-HS form} $($Definition 1.11, \cite{alessandrini1987closed}$)$.
    \item \textbf{p-SKT} $($or \textbf{p-pluriclosed}$)$ if it admits a $\partial\bar\partial$-closed weakly strictly positive $(p,p)$-form $\Omega$. Here, $\Omega$ is called a \textbf{p-SKT form} $($Theorem 2.4, \cite{alessandrini2011classes}$)$.
\end{enumerate}
\end{dfn}
\begin{rmk} \begin{enumerate}
    \item $1$-Kähler=Kähler, $1$-HS=HS and $1$-SKT=SKT.
    \item $(n-1)$-Kähler=Balanced, $(n-1)$-HS=sG, $(n-1)$-SKT=Gauduchon.
    \item \textbf{Astheno-Kähler} $\xLongrightarrow{\qquad}$ $(n-2)$-SKT (see §4 in \cite{jost1993nonlinear}). More generally, the notion of $p$-SKT defined in \cite{ivanov2013vanishing} implies the notion of $p$-SKT used in this article, but they are not equivalent in general.
    \item For every $1\leq p\leq n-1$, we have: $p$-Kähler $\xLongrightarrow{\qquad}$ $p$-HS $\xLongrightarrow{\qquad}$ $p$-SKT.
\end{enumerate}
\end{rmk}
Now, we define their associated hyperbolicity notions.
\begin{dfn} Let $X$ be a compact complex manifold with $\dim_\C X=n\geq 2$, and let $1\leq p\leq n-1$. We say that $X$ is:
\begin{enumerate}
    \item \textbf{p-Kähler hyperbolic} if it admits a $\widetilde{d}($bounded$)$ $p$-Kähler form $($Definition 3.1, \cite{haggui2023compact}$)$.
    \item \textbf{p-HS hyperbolic} if it admits a $\widetilde{d}($bounded$)$ $p$-HS form $($this was referred to as weakly $p$-Kähler hyperbolic in Definition 3.1, \cite{ma2024strongly}$)$.
    \item \textbf{p-SKT hyperbolic} if it admits a $(\widetilde{\partial+\bar\partial)}($bounded$)$ $p$-SKT form.
\end{enumerate}
\end{dfn}
\begin{rmk} 
\begin{enumerate} 
\item The case $p=1$ recovers the classical notions of Kähler hyperbolicity, HS hyperbolicity, and SKT hyperbolicity.
\item $(n-1)$-Kähler hyperbolic=Balanced hyperbolic, $(n-1)$-HS hyperbolic=sG hyperbolic and $(n-1)$-SKT hyperbolic=Gauduchon hyperbolic.
\item One can define \textit{astheno-Kähler hyperbolicity} in a similar way: We say that a compact astheno-Kähler manifold $X$ is \textbf{astheno-Kähler hyperbolic} if it admits an astheno-Kähler metric $\omega$ such that $\omega^{n-2}$ is ($\widetilde{\partial+\bar\partial}$)(bounded). Then:
\begin{center}
    Astheno-Kähler hyperbolic $\xLongrightarrow{\qquad}$ $(n-2)$-SKT hyperbolic
\end{center}
\item The hierarchy between $p$-Kähler, $p$-HS and $p$-SKT forms still holds in the hyperbolic context. Namely, for every $1\leq p\leq n-1$ we have: 
\begin{center}
$p$-Kähler hyperbolic $\xLongrightarrow{\qquad}$ $p$-HS hyperbolic $\xLongrightarrow{\qquad}$ $p$-SKT hyperbolic  
\end{center}
\item By arguments similar to those used in Propositions \ref{75}.(3) and \ref{76}, one can show that a Kähler hyperbolic manifold is $p$-Kähler hyperbolic, and every HS hyperbolic manifold is $p$-HS hyperbolic, respectively, for any $1\leq p\leq n-1$.
\end{enumerate} \label{82}
\end{rmk}
B.-L. Chen and X. Yang showed that a compact Kähler manifold homotopic to a compact negatively curved Riemannian manifold is Kähler hyperbolic (Theorem 1.1, \cite{chen2018compact}). This can be generalized to the $p$-Kähler hypebolic and $p$-HS hyperbolic cases for $1\leq p\leq n-1$, thanks to the following property:
\begin{prop}[Proposition 7.1, \cite{pansu1993introduction}, Lemma 3.2, \cite{chen2018compact}] Let $M$ be a complete simply connected Riemannian manifold with Riemannian sectional curvature $K\leq -c^2<0$. Then every closed bounded $k$-form on $M$ is d$($bounded$)$, for any $k\geq 2$. \label{77}
\end{prop}
Note that such a manifold $M$ cannot be compact since the Cartan–Hadamard theorem states that a complete simply connected Riemannian manifold with non-positive sectional curvature is diffeomorphic to $\R^{\dim M}$ via the exponential map at any point.
\begin{thm} Let $X$ be a compact $p$-Kähler $($resp. $p$-HS$)$ manifold with $\dim_\C X=n$, homotopic to a compact Riemannian manifold $(Y,g_Y)$ with negative Riemannian sectional curvature, with $1\leq p\leq n-1$. Then $X$ is $p$-Kähler hyperbolic $($resp. $p$-HS hyperbolic$)$. \label{79}
\end{thm}
\begin{proof} We will spell out the proof in the $p$-HS case, the $p$-Kähler case can be handled in a similar manner. Let $\Omega$ be a real $d$-closed $2p$-form on $X$ such that $\Omega^{p,p}$ is weakly strictly positive, i.e. $\Omega^{p,p}$ is a $p$-HS form. The goal is to show that $\Omega$ is $\widetilde{d}$(bounded). Let $\pi_X:\widetilde{X}\longrightarrow X$ and $\pi_Y:\widetilde{Y}\longrightarrow Y$ be the universal coverings of $X$ and $Y$ respectively. Since $X$ and $Y$ are homotopic, there exist two smooth maps $f_1:X\longrightarrow Y$ and $f_2:Y\longrightarrow X$ such that the induced map $(f_2\circ f_1)^*$ on $H^*_{DR}(X,\R)$ is the identity. Thus for any $\Omega'\in[\Omega]\in H^{2p}_{DR}(X,\R)$, there exists a ($2p-1$)-form $u$ in $X$ such that: 
\begin{equation}
    \Omega=\Omega'+du=(f_2\circ f_1)^*\Omega'+du=f_1^*(f_2^*\Omega')+du \label{78}
\end{equation}
Moreover, $f_2^*\Omega'$ is a $d$-closed $2p$-form on $Y$, hence $\widetilde{d}$(bounded) thanks to Proposition \ref{77}, that is $(\pi_Y^*\circ f_2^*)\Omega'=d\eta$ for some $\pi_Y^*g_Y$-bounded $(2p-1)$-form $\eta$ on $\widetilde{Y}$. Let $\widetilde{f_1}:\widetilde{X}\longrightarrow\widetilde{Y}$ and $\widetilde{f_2}:\widetilde{Y}\longrightarrow\widetilde{X}$ be the liftings of $f_1$ and $f_2$ respectively such that the following diagram:
\begin{equation*}
\begin{tikzcd}
\widetilde{X} 
  \arrow[r, "\widetilde{f_1}"] 
  \arrow[d, "\pi_X"'] 
& 
\widetilde{Y} 
  \arrow[r, "\widetilde{f_2}"] 
  \arrow[d, "\pi_Y"'] 
&
\widetilde{X}
\arrow [d, "\pi_X"]
\\
X 
  \arrow[r, "f_1"'] 
& 
Y 
  \arrow[r, "f_2"'] 
&
X 
\end{tikzcd}
\end{equation*}
commutes. Hence, by lifting the formula (\ref{78}) to the universal cover, we get:
\begin{align}
    \pi_X^*\Omega&=(\pi_X^*\circ f_1^*\circ f_2^*)(\Omega')+\pi_X^*(du)=\widetilde{f_1}^*\circ(\pi_Y^*\circ f_2^*)(\Omega')+d(\pi_X^*u) \\ &=\widetilde{f_1}^*(d\eta)+d(\pi_X^*u)=d(\widetilde{f_1}^*\eta+\pi_X^*u)
\end{align}
Fix a Riemannian metric $g_X$ on $X$. Then:
\begin{align}
    \lVert \widetilde{f_1}^*\eta+\pi_X^*u\rVert_{L^\infty_{\pi_X^*g_X}(\widetilde{X})}&\leq \lVert\widetilde{f_1}^*\eta\rVert_{L^\infty_{\pi_X^*g_X}(\widetilde{X})}+\lVert\pi_X^*u\rVert_{L^\infty_{\pi_X^*g_X}(\widetilde{X})}\\ 
    &\leq C\lVert f_1\rVert^{2p-1}_{\mathscr{C}^1_{g_X,g_Y}(X,Y)}\lVert\eta\rVert_{L^\infty_{\pi_X^*g_X}(\widetilde{X})}+\lVert u\rVert_{L^\infty_{g_X}(X)}<+\infty
\end{align}
since $X$ and $Y$ are compact (so $\pi_X$ and $\pi_Y$ are local isometries), where $C>0$ depends only on $X$ and $p$. Hence $\Omega$ is $\widetilde{d}$(bounded) and $X$ is $p$-HS hyperbolic.
\end{proof}
\begin{rmk}
\begin{enumerate}
    \item Theorem~\ref{79} is not relevant when the manifold $X$ itself admits a Hermitian metric $\omega$ with negative sectional curvature (for the \textbf{Chern connection}), since this implies the negativity of the first Chern-Ricci curvature form of $\omega$. Then, the ($1,1$)-form: 
    \begin{equation*}
    \eta:=-\mathrm{Ric}^{(1)}(\omega)=i\Theta_\omega(K_X)\underset{\mathrm{loc}}{=}i\partial\bar{\partial}\log(\omega^n).
    \end{equation*}
    defines a Kähler form on $X$. In particular, $X$ is \emph{Kähler hyperbolic} since ($X,\omega$) is diffeomorphic to ($X,\eta$) and are therefore homotopic. Hence Theorem~\ref{79} applies.
    \item The case $p=1$ in Theorem~\ref{79} is already of particular interest. Indeed, constructing a compact non-Kähler Hermitian–symplectic manifold that is homotopic to a compact Riemannian manifold with $K<0$ would provide an example of a non-Kähler Kobayashi hyperbolic manifold, assuming an affirmative answer to the question raised by J.~Streets and G.~Tian in Question~1.7 of \cite{streets2010parabolic} concerning the existence of compact non-Kähler Hermitian–symplectic manifolds. Such a scenario is, of course, widely regarded as unlikely in view of the Kobayashi conjecture, which predicts the ampleness of the canonical bundle $K_X$ for any compact Kobayashi hyperbolic manifold $X$. Nevertheless, the existence of a compact non-Kähler Kobayashi hyperbolic manifold would simply indicate that the Kähler assumption in the Kobayashi conjecture must be made explicit.
\end{enumerate}
\end{rmk}
\begin{pbm} 
Motivated by the search for non-Kähler Kobayashi hyperbolic manifolds, we ask whether Theorem~1.1 in \cite{broder2023general} can be strengthened in the non-Kähler setting: that is, whether a compact complex manifold carrying an SKT metric with negative holomorphic sectional curvature ($\text{HSC}<0$) must be SKT hyperbolic. Recall that any compact Hermitian manifold with $\text{HSC}<0$ is Kobayashi hyperbolic by \cite{kobayashi2005hyperbolic}, Theorem~4.1. As recently pointed out to the author by J.~Streets, the negativity of the \emph{real} bisectional curvature already forces a compact Hermitian manifold to be Kähler with ample canonical bundle (see \cite{lee2021complex}, Theorem~1.1), indicating that even slight strengthening of the "$\text{HSC}<0$" assumption may preclude genuinely non-Kähler examples. Note, for example, that a compact complex manifold with negative holomorphic bisectional curvature ($\mathrm{HBC}<0$) need not be K\"ahler SKT hyperbolic. Indeed, manifolds of the latter type have infinite fundamental group, since they are $L^2$-K\"ahler hyperbolic and therefore have generically large fundamental
group by Corollary~2.12 in~\cite{khelifati2025holomorphic}. In contrast, there exist
simply connected Kähler manifolds with $\mathrm{HBC}<0$, arising for instance
as complete intersections in complex projective manifolds \cite{mohsen2022construction}.

In \cite{kasuya2025partially}, Section §.5, it was observed that one does not need to assume negativity of the sectional curvature; instead, a weaker curvature-type condition introduced in the context of \emph{pluriclosed star split} manifolds \cite{popovici2023pluriclosed} suffices to obtain hyperbolicity results in several settings. This suggests the possibility of exploring whether analogous conditions could imply other forms of hyperbolicity for compact complex manifolds, possibly in a partial sense. For instance, one may ask whether such curvature-type hypotheses could lead to a partial hyperbolicity of Hermitian–Symplectic, SKT, strongly Gauduchon, or Gauduchon type. Since partial hyperbolicity is weaker and more flexible than full Kobayashi hyperbolicity, it could potentially allow for a wider range of non-Kähler examples, opening the door to new constructions and examples in this direction.
\end{pbm}
Building upon the ideas of Theorem~5.5 in~\cite{alessandrini2017product} and Proposition~3.5 in~\cite{haggui2023compact}, we derive the following results:
\begin{prop} Let $(X,\omega)$ be a compact Kähler manifold with $\dim_\C X=n$, and $Y$ be compact complex manifold with $\dim_\C Y=m$. Fix $1\leq q<m$.
\begin{enumerate}
    \item If $Y$ is $p$-HS $($resp. $p$-HS hyperbolic$)$ for all $q\leq p<m$, then $X\times Y$ is $(p+n)$-HS $($resp. $(p+n)$-HS hyperbolic$)$ for all $q\leq p<m$.
    \item If $Y$ is $p$-SKT $($resp. $p$-SKT hyperbolic$)$ for all $q\leq p<m$, then $X\times Y$ is $(p+n)$-SKT $($resp. $(p+n)$-SKT hyperbolic$)$ for all $q\leq p<m$.
\end{enumerate}
\end{prop}
\begin{proof} Let $\pi_X:\widetilde{X}\longrightarrow X$ and $\pi_Y:\widetilde{Y}\longrightarrow Y$ be the universal coverings of $X$ and $Y$, respectively. Let $\text{pr}_X:X\times Y\longrightarrow X$ and $\text{pr}_Y:X\times Y\longrightarrow Y$ be the projections onto $X$ and $Y$, respectively. The projections $\text{pr}_X$ and $\text{pr}_Y$ lift to projections $\widetilde{\text{pr}_X}:\widetilde{X}\times\widetilde{Y}\longrightarrow\widetilde{X}$ and $\widetilde{\text{pr}_Y}:\widetilde{X}\times\widetilde{Y}\longrightarrow\widetilde{Y}$. In particular, the following diagram commutes:
\begin{equation*}
\begin{tikzcd}
\widetilde{X} 
    \arrow[d, "\pi_X"']
  & 
\widetilde{X}\times\widetilde{Y} 
    \arrow[l, "\widetilde{\mathrm{pr}}_X"'] 
    \arrow[r, "\widetilde{\mathrm{pr}}_Y"] 
    \arrow[d, "\pi"] 
  &
\widetilde{Y}
    \arrow[d, "\pi_Y"]
\\
X 
  &
X\times Y 
    \arrow[l, "\mathrm{pr}_X"] 
    \arrow[r, "\mathrm{pr}_Y"'] 
  &
Y
\end{tikzcd}
\end{equation*}
\begin{enumerate}
    \item For $q\leq p<n$, we let $\Omega_p$ be a real $d$-closed $2p$-form on $Y$ such that $\Omega^{p,p}$ is weakly strictly positive, i.e. $\Omega^{p,p}$ is a $p$-HS form. Fix $q+n\leq j<m+n$ and let $h=\max\{0,j-m\}$. We set:
    \begin{equation}
        \Theta_j:=\underset{h\leq k\leq n}{\sum}\text{pr}_X^*\omega^k\wedge\text{pr}_Y^*\Omega_{j-k} \label{80}
    \end{equation}
    Then $\Theta_j$ is obviously a real $d$-closed $2j$-form on $X\times Y$. Thanks to Propositions 1.11 and 1.12 in \cite{demailly1997complex}, §.1, Chapter III, we infer that:
    \begin{equation}
        \Theta_j^{j,j}=\underset{h\leq k\leq n}{\sum}\text{pr}_X^*\omega^k\wedge\text{pr}_Y^*\Omega^{j-k,j-k}_{j-k}
    \end{equation}
    is weakly strictly positive. Hence $\Theta_j^{j,j}$ defines a $j$-HS form on $X\times Y$. If moreover, we assume the $p$-forms $\Omega_p$ are $\widetilde{d}$(bounded), i.e. $\forall q\leq p<m$, there exists a bounded ($2p-1$)-form $\eta_p$ in $\widetilde{Y}$ such that $\pi_Y^*\Omega_p=d\eta_p$. By pulling-back the expression (\ref{80}) to $\widetilde{X}\times\widetilde{Y}$, we get:
    \begin{align*}
        \pi^*\Theta_j&=\underset{h\leq k\leq n}{\sum}(\pi^*\circ\text{pr}_X^*)\omega^k\wedge(\pi^*\circ\text{pr}_Y^*)\Omega_{j-k} \\
        &=\underset{h\leq k\leq n}{\sum}\widetilde{\text{pr}_X}^*\widetilde{\omega}^k\wedge\widetilde{\text{pr}_Y}^*(\pi_Y^*\Omega_{j-k}) \\ &=\underset{h\leq k\leq n}{\sum}\widetilde{\text{pr}_X}^*\widetilde{\omega}^k\wedge d(\widetilde{\text{pr}_Y}^*\eta_{j-k}) \\
        &=d\left(\underset{h\leq k\leq n}{\sum}\widetilde{\text{pr}_X}^*\widetilde{\omega}^k\wedge \widetilde{\text{pr}_Y}^*\eta_{j-k}\right) \qquad\quad \text{since $\omega$ is $d$-closed}
    \end{align*}
    Hence for any $q+n\leq j<m+n$, $\Theta_j$ is $\widetilde{d}$(bounded).
    \item For $q\leq p<n$, we let $\Omega_p$ be a $p$-SKT form on $Y$, and we fix $q+n\leq j<m+n$. Set $h=\max\{0,j-m\}$, then the ($j,j$)-form:
    \begin{equation}
        \Theta_j:=\underset{h\leq k\leq n}{\sum}\text{pr}_X^*\omega^k\wedge\text{pr}_Y^*\Omega_{j-k} \label{81}
    \end{equation}
    defines a $j$-SKT form on $X\times Y$, again thanks to Propositions 1.11 and 1.12 in \cite{demailly1997complex}, §.1, Chapter III, and the fact $\omega$ is both $\partial$ and $\bar\partial$-closed. If moreover, the ($p,p$)-forms $\Omega_p$ are ($\widetilde{\partial+\bar\partial}$)(bounded), i.e. $\forall q\leq p<m$, there exist a bounded ($p-1,p$)-form $\alpha_p$ and a bounded ($p,p-1$)-form $\beta_p$ in $\widetilde{Y}$ such that $\pi_Y^*\Omega_p=\partial\alpha_p+\bar\partial\beta_p$, then:
    \begin{align*}
      \pi^*\Theta_j&=\underset{h\leq k\leq n}{\sum}(\pi^*\circ\text{pr}_X^*)\omega^k\wedge(\pi^*\circ\text{pr}_Y^*)\Omega_{j-k} \\
        &=\underset{h\leq k\leq n}{\sum}\widetilde{\text{pr}_X}^*\widetilde{\omega}^k\wedge\widetilde{\text{pr}_Y}^*(\pi_Y^*\Omega_{j-k}) \\ &=\underset{h\leq k\leq n}{\sum}\widetilde{\text{pr}_X}^*\widetilde{\omega}^k\wedge (\partial(\widetilde{\text{pr}_Y}^*\alpha_{j-k})+\bar\partial(\widetilde{\text{pr}_Y}^*\beta_{j-k})) \\
        &=\partial\left(\underset{h\leq k\leq n}{\sum}\widetilde{\text{pr}_X}^*\widetilde{\omega}^k\wedge \widetilde{\text{pr}_Y}^*\alpha_{j-k}\right)+\bar\partial\left(\underset{h\leq k\leq n}{\sum}\widetilde{\text{pr}_X}^*\widetilde{\omega}^k\wedge \widetilde{\text{pr}_Y}^*\beta_{j-k}\right) 
    \end{align*}
    Hence $\Theta_j$ is a $j$-SKT hyperbolic form on $X\times Y$, for all $q+n\leq j<m+n$.
\end{enumerate}
\end{proof}
The analogue of Proposition~3.8 in \cite{haggui2023compact} remains valid in the $p$-HS and $p$-SKT hyperbolic contexts. Namely:
\begin{prop} Let $X$ be a $p$-HS $($resp. $p$-SKT$)$ hyperbolic manifold with $\dim_\C X=n$, and $S$ be submanifold of $X$ with $\dim_\C S>p$. Then $S$ is $p$-HS $($resp. $p$-SKT$)$ hyperbolic too.
\end{prop}
\begin{proof} As the argument proceeds in the same manner as in the $p$-Kähler hyperbolic case, we shall restrict our attention to the $p$-SKT hyperbolic case, leaving the analogous proof for the $p$-HS hyperbolic case to the reader. Let $\iota:S\xhookrightarrow{\;\;\;\;\;\;}X$ be the natural inclusion, and let $\pi_X:\widetilde{X}\longrightarrow X$ and $\pi_S:\widetilde{S}\longrightarrow S$ be the universal coverings of $X$ and $S$ respectively. Then $\iota$ lifts to a natural inclusion on universal covers:
\begin{equation*}
\begin{tikzcd}
\widetilde{S}  
    \arrow[r, hookrightarrow, "\widetilde{\iota}"] 
    \arrow[d, "\pi_S"'] 
  &
\widetilde{X}
    \arrow[d, "\pi_X"]
\\
S 
\arrow[r, hookrightarrow, "\iota"']
  &
X 
\end{tikzcd}
\end{equation*}
Let $\Omega$ be a $p$-SKT hyperbolic form on $X$, that is: $\pi_X^*\Omega=\partial\alpha+\bar\partial\beta$, for some bounded ($p-1,p$) form $\alpha$ and bounded ($p,p-1$)-form $\beta$ on $\widetilde{X}$. Thanks to Proposition 1.12 in \cite{demailly1997complex}, §.1, Chapter III, $\iota^*\Omega$ defines a $p$-SKT form on $S$ and:
\begin{equation}
    \pi_S^*(\iota^*\Omega)=\widetilde{\iota}^*(\pi_X^*\Omega)=\partial(\widetilde{\iota}^*\alpha)+\bar\partial(\widetilde{\iota}^*\beta)
\end{equation}
Hence $\iota^*\Omega$ is a $p$-SKT hyperbolic form on $S$.
\end{proof} 
We now aim to highlight the relationship between the hyperbolicity notions defined in degree $2p$ and their corresponding geometric counterparts. Recall that in the case $p=1$, these are the notions of Kähler hyperbolicity, HS hyperbolicity, and SKT hyperbolicity (Remark~\ref{82}.1). They have been shown to imply Kobayashi hyperbolicity (see \cite{chen2018compact}, Theorem~4.1, Corollary~\ref{83} and \cite{marouani2023skt}, Theorem~3.5, respectively). 

For $p=n-1$, the corresponding notions are balanced hyperbolicity, sG hyperbolicity and Gauduchon hyperbolicity (Remark~\ref{82}.2). These manifolds have been proven to be divisorially hyperbolic (see \cite{marouani2023balanced}, Theorem~2.8; \cite{ma2024strongly}, Theorem~2.6; and \cite{marouani2023skt}, Theorem~3.10, respectively). 

For general $2\leq p \leq n-1$, we introduce a natural generalization of divisorial hyperbolicity (referred to as \emph{$p$-dimensional hyperbolicity} in \cite{haggui2023compact}, Definition~3.14 and \emph{$p$-hyperbolicity} in \cite{kasuya2025partially}, Definition 2.5).

Let $(X,\omega)$ be a compact Hermitian manifold with $\dim_\C X=n\geq 2$, and let $f:\C^p\longrightarrow X$ be a holomorphic map with $1\leq p\leq n-1$, which is \emph{non-degenerate} at some point $x_0\in\C^p$, i.e. $d_{x_0}f:\C^p\longrightarrow T^{1,0}_{x_0}X$ is of maximal rank. Set: $\Sigma_f:=\{x\in\C^p\hspace{1pt}:\hspace{1pt}f\text{ is degenerate at }x\}$; this is a proper analytic subset of $\C^p$. Then $f^*\omega$ defines a \emph{degenerate metric} on $\C^p$ with degeneration set $\Sigma_f$ ; that is $f^*\omega\geq 0$ on $\C^p$ and defines a genuine Hermitian metric on $\C^p\setminus\Sigma_f$. For any $r>0$, we define the ($\omega,f$)\textbf{-volume} of the ball $\B_r\subset\C^p$ of radius $r$ centered at $0$ to be:
\begin{equation}
    \text{Vol}_{\omega,f}(\B_r):=\int_{\B_r}f^*\omega_p>0\hspace{2pt},\quad\qquad\text{where }\hspace{2pt}\omega_p=\dfrac{\omega^p}{p!}
\end{equation}
For any $z\in\C^p$, let $\beta(z)=\lvert z\rvert^2$ be its squared Euclidean norm. By definition of the Hodge star operator induced by the Hermitian metric $f^*\omega$ on $\C^p\setminus\Sigma_f$, we have: 
\begin{equation}
    \dfrac{d\beta(z)}{\lvert d\beta(z)\rvert_{f^*\omega}}\wedge*_{f^*\omega}\left(\dfrac{d\beta(z)}{\lvert d\beta(z)\rvert_{f^*\omega}}\right)=f^*\omega_p(z)\hspace{2pt},\qquad\forall z\in\C^p\setminus\Sigma_f
\end{equation}
This implies that the ($2p-1$)-form $d\sigma_{\omega,f}:=*_{f^*\omega}\left(\dfrac{d\beta(z)}{\lvert d\beta(z)\rvert_{f^*\omega}}\right)$ defines an area measure on spheres $\Sp_r:\partial\B_r\subset\C^p$. 

In particular, we define the ($\omega,f$)\textbf{-area} of the sphere $\Sp_r$ of radius $r>0$ by:
\begin{equation}
    A_{\omega,f}(\Sp_r):=\int_{\Sp_r}d\sigma_{\omega,f,r}>0\hspace{2pt},\qquad\text{ where }\hspace{2pt}d\sigma_{\omega,f,r}=\left(d\sigma_{\omega,f}\right)_{\lvert\Sp_r}
\end{equation}
\begin{dfn}[Definition 2.3, \cite{marouani2023balanced} and Definition 2.3, \cite{kasuya2025partially}]
In the same setting, we say that $f$ has \textbf{subexponential growth} if:
    \renewcommand{\labelenumi}{(\textit{\roman{enumi}})}
    \begin{enumerate}
    \item $\exists\hspace{1pt}C_0,r_0>0$ such that: $\int_{\Sp_r}\lvert d\beta\rvert_{f^*\omega}d\sigma_{f,\omega,r}\leq C_0 r \text{Vol}_{\omega,f}(\B_r)$, $\forall\hspace{1pt}r>r_0$.
    \item $\forall\hspace{1pt}\lambda>0$, we have: $\underset{t\rightarrow+\infty}{\limsup}\left(\dfrac{t}{\lambda}-\log\left(\int_0^t\text{Vol}_{\omega,f}(\B_r)dr\right)\right)=+\infty$.
    \end{enumerate} \label{84}
\end{dfn}
\begin{rmk} Since $X$ is compact, any two Hermitian metrics on $X$ are equivalent. Therefore, the subexponential growth condition on the holomorphic maps from some $\C^p$ into $X$ is independent of the choice of the Hermitian metric on $X$.
\end{rmk}
\begin{dfn}[Definition 2.5, \cite{kasuya2025partially} and Definition 3.14, \cite{haggui2023compact}] 
We say that $X$ is $p$\textbf{-cyclically hyperbolic} for some $1\leq p\leq n-1$ if there is no holomorphic map $f:\C^p\longrightarrow X$ that is non-degenerate at some point $x_0\in\C^p$ and has subexponential growth in the sense of the previous Definition.
\end{dfn}
\begin{rmk}
    \begin{enumerate}
        \item The case $p=n-1$ corresponds exactly to the definition of \textbf{divisorial hyperbolicity} (Definition 2.7, \cite{marouani2023balanced}). 
        \item When $p=1$, it is clear that Kobayashi hyperbolicity implies $1$-cyclic hyperbolicity. Conversly, assume $X$ is a compact complex manifold which admits a non-constant holomorphic map $f:\C\longrightarrow X$. Thanks to \textbf{Brody’s Renormalisation Lemma} (\cite{brody1978compact}, Lemma 2.1), we can assume $f$ to be of order $2$ (see \cite{lang1987introduction}, Theorem 2.6-p.72), that is:
        \begin{equation*}
            \text{Vol}_{\omega,f}(\B_r):=\int_{\B_r}f^*\omega\underset{r\rightarrow\infty}{\sim}r^2
        \end{equation*}
        for some (hence any by compactness of $X$) Hermitian metric $\omega$ on $X$. Or equivalently:
        \begin{equation}
            \dfrac{1}{C_1}r^2-C_2\leq\text{Vol}_{\omega,f}(\B_r)\leq C_1 r^2+C_2\hspace{2pt},\qquad\forall r>r_0
        \end{equation}
        for some constants $C_1,C_2,r_0>0$. One readily checks that such an $f$ satisfies condition (\emph{ii}) in Definition \ref{84}. On the other hand, one can write $f^*\omega$ as:
        \begin{equation*}
            f^*\omega=\dfrac{i}{2}\mu(z)dz\wedge d\bar z
        \end{equation*}
        where $\mu:\C\longrightarrow\R$ is a $\mathscr{C}^\infty$-function, positive on $\C\setminus\Sigma_f$ and vanishes along $\Sigma_f$ (which is a countable set of points). For any $z=r(\cos\theta+i\sin\theta)\in\C\setminus\Sigma_f$, simple computations yield:
        \begin{equation*}
            \beta(z)=\lvert z\rvert^2\Longrightarrow d\beta(z)=2\lvert z\rvert d\lvert z\rvert\Longrightarrow\lvert d\beta(z)\rvert_{f^*\omega}=\dfrac{2\lvert z\rvert}{\sqrt{\mu(z)}}
        \end{equation*}
        Hence, the length measure induced by $f$ and $\omega$ is given in polar coordinates by:
        \begin{equation*}
            \sigma_{\omega,f}=*_{f^*\omega}\left(\dfrac{d\beta}{\lvert d\beta\rvert_{f^*\omega}}\right)=r\sqrt{\mu}\hspace{1pt}d\theta
        \end{equation*}
        Therefore, for any $r>0$ we get:
        \begin{equation*}
          \int_{\Sp_r}\lvert d\beta\rvert_{f^*\omega}d\sigma_{\omega,f,r}=\int_0^{2\pi}\dfrac{2r}{\sqrt{\mu}}r\sqrt{\mu}d\theta=4\pi r^2
        \end{equation*}
        which is clearly dominated by $r\text{Vol}_{\omega,f}(\B_r)$ for $r>r_0$ large enough. So $f$ satisfies also condition (\emph{i}) in Definition \ref{84}. Thus, $1$-cyclic hyperbolicity coincides with Brody hyperbolicity (equivalently, with Kobayashi hyperbolicity). 
        \item Kobayashi hyperbolicity implies $p$-cyclic hyperbolicity, $\forall 1\leq p\leq n-1$. Indeed, if $f:\C^p\longrightarrow X$ is a holomorphic map, non-degenerate at some point $x\in\C^p$, then $f$ remains non-constant along every complex line passing through $x$.
    \end{enumerate}
\end{rmk}
\begin{pbm}
    Does $p$-cyclic hyperbolicity imply ($p+1$)-cyclic hyperbolicity ?
\end{pbm}
S. Marouani and F. Haggui established that every $p$-Kähler hyperbolic manifold is $p$-cyclically hyperbolic (Theorem 3.15 in \cite{haggui2023compact}). Regarding $p$-HS hyperbolicity, Y. Ma showed that this condition also implies $p$-cyclic hyperbolicity in \cite{ma2024strongly}, Theorem 3.2. We now prove that the same holds for $p$-SKT hyperbolic manifolds:
\begin{thm}
Every compact $p$-SKT hyperbolic manifold is $p$-cyclically hyperbolic.
\end{thm}
\begin{proof}
Let ($X,\omega$) be a compact Hermitian manifold with $\dim_\C X=n$ and $\pi:\widetilde{X}\longrightarrow X$ be the universal cover of $X$. Assume $X$ equipped with a $p$-SKT hyperbolic form $\Omega$, i.e. $\Omega$ is a smooth weakly strictly positive ($p,p$)-form such that $\partial\bar\partial\Omega=0$ and $\widetilde{\Omega}=\partial\alpha+\bar\partial\gamma$, where $\alpha$ (resp. $\gamma$) is a smooth $\widetilde{\omega}$-bounded ($p-1,p$)-form (resp. ($p,p-1$)-form) on $\widetilde{X}$. Assume that a holomorphic map $f:\C^p\longrightarrow X$ such as in Definition~\ref{84} exists. Then $f^*\omega$ defines a degenerate Hermitian metric on $X$ (with degeneracy locus $\Sigma_f$). Let us first make the following observation:
\begin{obs} Let $\eta$ be an $\widetilde{\omega}$-bounded $k$-form on $\widetilde{X}$, with $k\leq 2p$. Then $\widetilde{f}^*\eta$ is $f^*\omega$-bounded on $\C^p$. \label{85}
\end{obs}
\begin{proof} Since $\C^p$ is simply connected, $f$ lifts to a holomorphic map $\widetilde{f}:\C^p\longrightarrow\widetilde{X}$, that is: $\pi\circ\widetilde{f}=f$. In particular, the degeneracy locus of $\widetilde{f}$ is also $\Sigma_f$. Let $\eta$ be a $\widetilde{\omega}$-bounded $k$-form on $\widetilde{X}$, with $k\leq 2p$. For any $x\in\C^p$, and any tangent vectors $\xi_1,\cdots,\xi_k\in\C^p$ at $x$, we have:
\begin{align*}
\lvert\widetilde{f}^*\eta(\xi_1,\cdots,\xi_k)\rvert^2_{f^*\omega}&=\lvert\eta(\widetilde{f}_*(\xi_1),\cdots,\widetilde{f}_*(\xi_k))\rvert^2_{\widetilde{\omega}} \\
&\leq C\lvert\widetilde{f}_*(\xi_1)\rvert^2_{\widetilde{\omega}}\cdots\lvert\widetilde{f}_*(\xi_k)\rvert^2_{\widetilde{\omega}} \qquad\qquad\text{since }\eta\text{ is bounded} \\
&=C\lvert\xi_1\rvert^2_{f^*\omega}\cdots\lvert\xi_k\rvert^2_{f^*\omega}\qquad\qquad\qquad\text{since }\widetilde{f}^*\widetilde{\omega}=f^*\omega
\end{align*} 
\end{proof}
Now we can resume the proof of our theorem. $f^*\Omega$ defines a positive ($p,p$)-form (hence \emph{strongly positive} by maximality of the dimension, see Corollary 1.9 in \cite{demailly1997complex}, §.1, Chapter III) on $\C^p\setminus\Sigma_f$. Hence, there exists a smooth positive function $h:\C^p\setminus\Sigma_f\longrightarrow\R$ such that: $f^*\omega_p=hf^*\Omega$, where $\omega_p=\dfrac{\omega^p}{p!}$. But since both $f^*\omega_p$ and $f^*\Omega$ are smooth, $h$ extends smoothly to the whole $\C^p$. In particular, we have for any $z\in\C^p$:
\begin{equation*}
    \sqrt{\binom{n}{p}}=\lvert\omega_p(f(z))\rvert_\omega=\lvert f^*\omega_p(z)\rvert_{f^*\omega}=h(z)\lvert f^*\Omega(z)\rvert_{f^*\omega}=h(z)\lvert\Omega(f(z))\rvert_\omega
\end{equation*}
Thus: $h(z)=\displaystyle\sqrt{\binom{n}{p}}\lvert\Omega(f(z))\rvert^{-1}_\omega$, $\forall z\in\C^p$. On the other hand, since $\Omega(x)\neq 0$ for any $x\in X$ by positivity, the compactness of $X$ implies that there exists a constant $A>0$ such that: $\dfrac{1}{A}\leq\lvert\Omega(x)\rvert_\omega\leq A$. Hence, we get for any $r>0$:
\begin{equation}
    \text{Vol}_{\omega,f}(\B_r)=\int_{\B_r}f^*\omega_p=\int_{\B_r}hf^*\Omega\leq A\sqrt{\binom{n}{p}}\int_{\B_r}f^*\Omega
    \label{86}
\end{equation}
Moreover, we have:
\begin{equation}
    f^*\Omega=\widetilde{f}^*\widetilde{\Omega}=\widetilde{f}(\partial\alpha+\bar\partial\gamma)\overset{(a)}{=}\widetilde{f}(d(\alpha+\gamma))=d\widetilde{f}(\alpha+\gamma) \label{87}
\end{equation}
where ($a$) is a consequence of the bidegree constraints. Substituting (\ref{87}) into (\ref{86}) gives:
\begin{equation}
    \text{Vol}_{\omega,f}(\B_r)\leq A\sqrt{\binom{n}{p}}\int_{\B_r}d\widetilde{f}^*(\alpha+\gamma)=A\sqrt{\binom{n}{p}}\int_{\Sp_r}\widetilde{f}^*(\alpha+\gamma)\overset{(b)}{\leq}BA_{\omega,f}(\Sp_r)\hspace{2pt},
    \label{88}
\end{equation}
where ($b$) follows from Observation \ref{85} (here $\eta=\alpha+\gamma$). On the other hand, the Cauchy-Schwarz inequality yields:

\begin{equation}
    A_{\omega,f}(\Sp_r)^2\leq\left(\int_{\Sp_r}\dfrac{1}{\lvert d\beta\rvert_{f^*\omega}}d\sigma_{\omega,f,r}\right)\left(\int_{\Sp_r}\lvert d\beta\rvert_{f^*\omega}d\sigma_{\omega,f,r}\right)\hspace{2pt},\qquad\forall r>0 \label{89}
\end{equation}
where $\beta(z)=\lvert z\rvert^2=r^2$, hence $d\beta=2rdr$. Therefore, by combining (\ref{88}) and (\ref{89}), we get for any $r>r_0$ sufficiently large:
\begin{align*}
    \text{Vol}_{\omega,f}(\B_r)&=\int_0^r\left(\int_{\Sp_t}\dfrac{1}{\lvert d\beta\rvert_{f^*\omega}}d\sigma_{\omega,f,t}\right)d\beta \\
    &\overset{(\ref{89})}{\geq} \int_0^r\dfrac{A_{\omega,f}(\Sp_t)^2}{\int_{\Sp_t}\lvert d\beta\rvert_{f^*\omega}d\sigma_{\omega,f,t}}2tdt \\ &\overset{(\ref{88})}{\geq}\dfrac{2}{B^2}\int_0^r\dfrac{t\text{Vol}_{\omega,f}(\B_t)}{\int_{\Sp_t}\lvert d\beta\rvert_{f^*\omega}d\sigma_{\omega,f,t}}\text{Vol}_{\omega,f}(\B_t)dt \\
    & \overset{(c)}{\geq}\dfrac{2}{C_0B^2}\int_0^r\text{Vol}_{\omega,f}(\B_t)dt=:CF(r)
\end{align*}
where $C:=\dfrac{2}{C_0B^2}>0$, $F(r):=\int_0^r\text{Vol}_{\omega,f}(\B_t)dt$ and ($c$) follows from the condition ($i$) in Definition~\ref{84}. Thus, the previous inequality takes the form:
\begin{equation*}
    F'(r)\geq CF(r) \xLongrightarrow{\qquad} (\log F(r))'\geq C \xLongrightarrow{\qquad} F(r)\geq F(r_0)\exp (C(r-r_0))\hspace{2pt},\qquad\forall r>r_0
\end{equation*}
In this situation, condition ($ii$) in Definition~\ref{84} forces $F(r)=0$, $\forall r>r_0$ which in turn yields: $\text{Vol}_{\omega,f}(\B_r)=0$, $\forall r> r_0$. Therefore $f^*\omega=0$, contradicting our assumption.
\end{proof} 

\subsection{Deformation Stability Results}
Finally, we will study the holomorphic deformations of $p$-HS hyperbolic and $p$-Kähler hyperbolic manifolds. Our first result in this direction is the deformation \textbf{openness} of the $p$-HS hyperbolicity property:
\begin{thm} Let $\{X_t:t\in\B_\varepsilon(0)\}$ be an analytic family of compact complex $n$-dimensional manifolds such that $X_0$ is $p$-HS hyperbolic for some $1\leq p\leq n-1$. Then the fibers $X_t$ are also $p$-HS hyperbolic for $t$ sufficiently small.
\end{thm}
\begin{proof}
   Fix a $\mathscr{C}^\infty$-family $\{\omega_t\}_{t\in\B_\varepsilon(0)}$ of Hermitian metrics on $(X_t)_{t\in\B_\varepsilon(0)}$. Let $\Omega$ be a $d$-closed $2p$-form on $X_0$ such that $\Omega^{p,p}$ is weakly strictly positive and $\widetilde{\Omega}=d\Gamma$, for some $\widetilde{\omega_0}$-bounded ($2p-1$)-form on $\widetilde{X_0}$. For $t\in\B_\varepsilon(0)$, let $\Omega_t^{p,p}$ be the ($p,p$)-component of $\Omega$ with respect to the complex structure on $X_t$ (In particular: $\Omega_0^{p,p}=\Omega^{p,p}$, which is weakly strictly positive by hypothesis). By continuity, $\Omega_t^{p,p}$ is weakly strictly positive for $t$ close enough to 0 (see for instance Lemma 4.2, \cite{bellitir2020deformation} or Proposition 4.12, \cite{rao2022power} for the proof), meaning that $\Omega_t^{p,p}$ is a $p$-HS form on $X_t$, for $t\sim 0$. Morevoer, $\widetilde{\Omega}=d\Gamma$ with: $\lVert\Gamma\rVert_{L^\infty_{\widetilde{\omega_t}}(\widetilde{X_t})}\leq C_t\lVert\Gamma\rVert_{L^\infty_{\widetilde{\omega_0}}(\widetilde{X_0})}<\infty$, for some $C_t>0$ (since the metrics $\omega_0$ and $\omega_t$ are comparable on the underlying \textbf{compact} manifold $X$). Hence $X_t$ is also $p$-HS hyperbolic for $t$ close enough to 0.
\end{proof}
Since $p$-Kähler hyperbolicity implies $p$-HS hyperbolicity for any $1\leq p\leq n-1$, it follows that small deformations of a compact $p$-Kähler hyperbolic manifold are $p$-HS hyperbolic. In particualr, in the case $p=1$, we have:
\begin{cor} 
Let $\{X_t:t\in\B_\varepsilon(0)\}$ be an analytic family of compact complex manifolds of dimension $n$, such that $X_0$ is Kähler hyperbolic. Then the fibers $X_t$ are HS-hyperbolic for $t$ sufficiently small.  
\end{cor}
When $2\leq p\leq n-1$, the fibers are not necessarily $p$-Kähler. Sufficient conditions ensuring the stability of the $p$-Kähler property were given for instance in \cite{rao2019local}, Theorem 1.1 and \cite{rao2022power}, Theorem 1.8. In order to study the small deformations of $p$-Kähler hyperbolic manifolds, we will assume that the $\partial\bar\partial$-property is satisfied. Note that a $p$-Kähler $\partial\bar\partial$-manifold with $2\leq p\leq n-1$ is not necessarily Kähler (take a proper modification of a compact Kähler manifold, see Theorem 4.8 in \cite{alessandrini1993metric} and Theorem 5.22 in \cite{deligne1975real}). The following result extends part~(1) of Theorem~\ref{71} as well as Theorem~4.10 of \cite{khelifati2025holomorphic} in the $L^\infty$-setting, for arbitrary degree $1 \leq p \leq n-1$:
\begin{thm} Let $\{X_t:t\in\B_\varepsilon(0)\}$ be an analytic family of compact complex $n$-dimensional manifolds such that $X_0$ is a $p$-Kähler hyperbolic $\partial\bar\partial$-manifold for some $1\leq p\leq n-1$. Then there exists a $\mathscr{C}^\infty$-family of $p$-Kähler forms $\{\Omega_t\}_{t\sim0}$ on the fibers $X_t$ that are $p$-SKT hyperbolic for $t$ sufficiently small.
\end{thm}
\begin{proof} Let $\{\omega_t\}_{t\in\B_\varepsilon(0)}$ be a $\mathscr{C}^\infty$-family of Hermitian metrics on the fibers $X_t$, and let $\Omega$ be a $p$-Kähler hyperbolic form on $X_0$, that is $\widetilde{\Omega}=d\eta$ for some $\widetilde{\omega_0}$-bounded ($2p-1$)-form $\eta$ on $\widetilde{X_0}$. Recall that the Bott-Chern Laplacian $\Delta_{BC}$ (with respect to the metric $\omega_0$), being elliptic and formally self-adjoint, induces the following $L^2_{\omega_0}$-orthogonal three-space decomposition in any bidegree ($p,q$) (see \cite{popovici2015aeppli}, §2):
\begin{equation}
    \mathscr{C}^\infty_{p,q}(X_0)=\mathcal{H}_{BC}^{p,q}(X_0)\oplus\text{Im}\hspace{2pt}\partial\bar\partial\oplus\left(\text{Im}\hspace{2pt}\partial^*+\text{Im}\hspace{2pt}\bar\partial^*\right)\quad, \label{73}
\end{equation}
and we have
\begin{equation}
    \text{Ker}\hspace{2pt}\bar\partial\cap\text{Ker}\hspace{2pt}\partial=\mathcal{H}_{BC}^{p,q}(X_0)\oplus\text{Im}\hspace{2pt}\partial\bar\partial \label{74}
\end{equation}
yielding the Hodge isomorphism $\mathcal{H}_{BC}^{p,q}(X_0)\simeq H_{BC}^{p,q}(X_0)$. Now, since $\Omega$ is real (as $\Omega$ is weakly strongly positive), $d\Omega=0$ implies that $\partial\Omega=\bar\partial\Omega=0$. Hence $\Omega\in\mathcal{H}_{BC}^{p,p}(X_0)\oplus\text{Im}\hspace{2pt}\partial\bar\partial$ thanks to (\ref{74}). Moreover, $\Omega=F^{p,p}_{BC}\Omega+\Delta_{BC}\G_{BC}\Omega$ by definition of the Green operator $\G_{BC}$, where $F^{p,p}_{BC}\Omega$ is the projection of $\Omega$ onto the Bott-Chern harmonic space $\mathcal{H}^{p,p}_{BC}(X_0)$. 

The uniqueness of the decompostion of $\Omega$ with respect to (\ref{73}) yields: $\Omega=F^{p,p}_{BC}\Omega+\partial\bar\partial\alpha$, for some $\alpha\in\mathscr{C}^\infty_{p-1,p-1}(X_0)$. We can take $\alpha$ to be the minimal $L^2_{\omega_0}$-norm solution of this equation, which is given by the following Neumann-type formula (\cite{popovici2015aeppli}, Theorem 4.1):
\begin{equation*}
    \alpha=(\partial\bar\partial)^*\G_{BC}(\Omega-F_{BC}^{p,p}\Omega)=(\partial\bar\partial)^*\G_{BC}\Omega
\end{equation*}
In other words, one can write $\Omega$ as: $\Omega=(F_{BC}^{p,p}+\partial\bar\partial\bar\partial^*\partial^*\G_{BC})\Omega$. We decompose $\Omega$ as: $\Omega=\underset{j+k=2p}{\sum}\Omega_t^{k,j}$. Since $X_0$ ia a $\partial\bar\partial$-manifold, the function $t\in\B_\varepsilon(0)\longmapsto h^{p,p}_{BC}(X_t)$ is constant (after shrinking $\varepsilon$ if necessary, see \cite{angella2013lemma}, Corollary 3.7). Then one defines:
\begin{equation}
    \Omega_t:=(F_{BC,t}^{p,p}+\partial_t\bar\partial_t\bar\partial_t^*\partial_t^*\G_{BC,t})\Omega_t^{p,p}\hspace{2pt},\qquad\forall t\in\B_\varepsilon(0) \label{92}
\end{equation}
   where $\G_{BC,t}:\mathscr{C}_{p,p}^\infty(X_t)\longrightarrow\mathscr{C}_{p,p}^\infty(X_t)$ is the Green operator associated with the Bot-Chern Laplacian $\Delta_{BC,t}$, and $F_{BC,t}^{p,p}:\mathscr{C}^\infty_{p,p}(X_t)\longrightarrow\mathcal{H}_{BC}^{p,p}(X_t)$ is the projection onto the Bott-Chern harmonic ($p,p$)-forms. $\{\Omega_t\}_{t\in\B_\varepsilon(0)}$ defines a $\mathscr{C}^\infty$-family of $d$-closed ($p,p$)-forms (thanks to Theorem 5 in \cite{kodaira1960deformations}) with $\Omega_0=\Omega$. In particular, $\{\Omega_t\}_{t\sim 0}$ is a $\mathscr{C}^\infty$-family of $p$-Kähler forms on the respective fibers $X_t$ for $t\sim 0$, with $\Omega_0=\Omega$. 
   
   Now, $d\Omega=0$ implies: $\partial_t\bar\partial_t\Omega_t^{k,j}=0$, with $k+j=2p$, $\forall t\sim 0$. Using the fact that the canonical isomorphism given in Proposition~\ref{94}.(3) is induced by the identity map (which exists thanks to the deformation openness of the $\partial\bar\partial$-property), we get:
   
     \begin{equation*}
         [\Omega_t^{k,j}]_A=[F_{BC,t}^{k,j}\Omega_t^{k,j}]_A\quad,\qquad\text{ with }k=j=2p\quad,\quad\forall t\sim 0.
     \end{equation*}
     as $[\Omega_t^{k,j}]_{BC}=[F_{BC,t}^{k,j}\Omega_t^{k,j}]_{BC}$. This implies that: $\Omega_t^{k,j}=F_{BC,t}^{k,j}\Omega_t^{k,j}+\partial_t u_t^{k-1,j}+\bar\partial_t u_t^{k,j-1}$, with $k+j=2p$, $t\sim 0$, and $u_t^{k-1,j}\in\mathscr{C}^\infty_{k-1,j}(X_t)$, $u_t^{k,j-1}\in\mathscr{C}^\infty_{k,j-1}(X_t)$. 
     
     Hence $\Omega$ decomposes with respect to the complex structure on $X_t$ for each $t\in\B_\varepsilon(0)$ as:
\begin{equation}
    \Omega=\underset{j+k=2p}{\sum}F_{BC,t}^{k,j}\Omega_t^{k,j}+\partial_t v_t+\bar\partial_t w_t\hspace{2pt}\quad,\qquad v_t,w_t\in\mathscr{C}_{2p-1}^\infty(X_t)\quad,\quad\forall t\sim 0, \label{93}
\end{equation} 
where $F_{BC,t}^{k,j}:\mathscr{C}^\infty_{k,j}(X_t)\longrightarrow\mathcal{H}_{BC}^{k,j}(X_t)$ is the projection onto the Bott-Chern harmonic ($k,j$)-forms.  By lifting (\ref{93}) to the universal cover and taking the ($p,p$)-component we get:
\begin{equation}
    \widetilde{F_{BC,t}^{p,p}\Omega_t^{p,p}}=\partial_t(\eta_t^{p-1,p}-\pi^*v_t^{p-1,p})+\bar\partial_t(\eta_t^{p,p-1}-\pi^*w_t^{p,p-1})\hspace{2pt},\qquad t\sim 0.
\end{equation}
Hence, lifting (\ref{92}) yields for $t$ sufficiently small:
\begin{equation}
    \widetilde{\Omega_t}=\partial_t(\eta_t^{p-1,p}+\pi^*(\bar\partial_t\bar\partial_t^*\partial_t^*\G_{BC,t}\Omega_t^{p,p}-v_t^{p-1,p})+\bar\partial_t(\eta_t^{p,p-1}-\pi^*w_t^{p,p-1})=\partial_t\alpha_t+\bar\partial_t\beta_t
\end{equation}
where $\alpha_t:=\eta_t^{p-1,p}+\pi^*(\bar\partial_t\bar\partial_t^*\partial_t^*\G_{BC,t}\Omega_t^{p,p}-v_t^{p-1,p})$ and $\beta_t=\eta_t^{p,p-1}-\pi^*w_t^{p,p-1}$, with:
\begin{align*}
\begin{cases}
    \lVert\alpha_t\rVert_{L^\infty_{\widetilde{\omega_t}}(\widetilde{X}_t)}&\leq \lVert\eta_t^{p-1,p}\rVert_{L^\infty_{\widetilde{\omega_t}}(\widetilde{X}_t)}+\lVert\pi^*(\bar\partial_t\bar\partial_t^*\partial_t^*\G_{BC,t}\Omega_t^{p,p}-v_t^{p-1,p})\rVert_{L^\infty_{\widetilde{\omega_t}}(\widetilde{X}_t)}\\ &\leq C_t\lVert\eta\rVert_{L^\infty_{\widetilde{\omega_0}}(\widetilde{X})}+\lVert \bar\partial_t\bar\partial_t^*\partial_t^*\G_{BC,t}\Omega_t^{p,p}\rVert_{L^\infty_{\omega_t}(X_t)}+\lVert v_t\rVert_{L^\infty_{\omega_t}(X_t)}<\infty \\
    \lVert\beta_t\rVert_{L^\infty_{\widetilde{\omega_t}}(\widetilde{X}_t)}&\leq \lVert\eta_t^{p,p-1}\rVert_{L^\infty_{\widetilde{\omega_t}}(\widetilde{X}_t)}+\lVert\pi^* w_t^{p,p-1}\rVert_{L^\infty_{\widetilde{\omega_t}}(\widetilde{X}_t)}\leq C_t\lVert\eta\rVert_{L^\infty_{\widetilde{\omega_0}}(\widetilde{X})}+\lVert w_t\rVert_{L^\infty_{\omega_t}(X)}<\infty
\end{cases}
\end{align*}
for some constant $C_t>0$ depending on the metric $\omega_t$. Hence $\Omega_t$ is $p$-SKT hyperbolic.
\end{proof}
Regarding the deformation behavior of $p$-SKT hyperbolicity, it is already known that in degree $p=1$, the SKT property is \textbf{not} open under small deformations of the complex structure (see, for instance, \cite{tardini2017geometric}, Remark~4.6). However, under the $\partial\bar{\partial}$-assumption, the $p$-SKT property becomes stable under small deformations. Indeed, if $X$ is a compact $\partial\bar\partial$-manifold, then (see \cite{alessandrini2011classes}, Corollary~3.4, or \cite{bellitir2020deformation}, Theorem~4.3):
\begin{center}
    $X$ is $p$-SKT $\xLeftrightarrow{\qquad}$ $X$ is $p$-HS \hspace{2pt},\qquad where  $1\leq p\leq n-1$.
\end{center}
The $\partial\bar\partial$-assumption appears to be quite strong, since, as mentioned above, D.~Popovici conjectured that every $\partial\bar\partial$-manifold is balanced, and the Fino-Vezzoni conjecture predicts that any balanced SKT manifold is in fact Kähler (\cite{fino2015special}, Problem 3). This suggests a deeper relationship between these structures. Motivated by this, we propose the following conjecture for any $1\leq p\leq n-1$:
\begin{conj}
    Let $X$ be a compact $\partial\bar\partial$-manifold. Then:
    \begin{center}
        $X$ is $p$-SKT $\xLeftrightarrow{\qquad}$ $X$ is $p$-Kähler \hspace{2pt},\qquad where $1\leq p\leq n-1$.
    \end{center}
\end{conj}
A weaker criterion ensuring the stability of the $p$-SKT property under small deformations may be extracted from the arguments in \cite{rao2022power}, Proposition~4.1 and \cite{cavalcanti2012hodge}, Theorem 8.11.

The preceding discussion, in conjunction with Conjecture~4.22 and Conjecture~4.25 in \cite{khelifati2025holomorphic}, leads naturally to the following problem for degree $1 \leq p \leq n-1$:
\begin{pbm} Let $X$ be a compact $\partial\bar\partial$-manifold and fix $1\leq p\leq n-1$. We ask whether the following assertion holds:
    \begin{center}
        $X$ is $p$-SKT hyperbolic $\xLeftrightarrow{\qquad}$ $X$ is $p$-HS hyperbolic $\xLeftrightarrow{\qquad}$ $X$ is $p$-Kähler hyperbolic
    \end{center}
\end{pbm}
An affirmative answer to this question would imply the deformation openness of $p$-SKT hyperbolicity, provided the $\partial\bar{\partial}$-property holds.thereby answering positively Question 5.6 in [52]. It follows from Theorem 3.2.17.(2)

\section*{Acknowledgements}
This work is part of the  author’s PhD thesis, and he would like to thank his supervisor Dan Popovici for carefully reading the manuscript and for providing several very useful remarks. The author would also like to express his gratitude to Philippe Eyssidieux for pointing out the paper of Jean-Paul Mohsen on the construction of negatively curved complete intersections \cite{mohsen2022construction}. Many thanks are also due to Jeffrey Streets for drawing his attention to his very interesting joint work with Man-Chun Lee \cite{lee2021complex}.

\bibliographystyle{acm}
\bibliography{DCH}
\end{document}